\documentclass[a4paper,11pt]{amsart}
\usepackage{amssymb}
\usepackage{amscd}
\usepackage{amsmath,amsthm}
\usepackage[colorlinks=true]{hyperref}
\usepackage{enumerate}
\usepackage{booktabs,multirow}
\usepackage{tikz}
\usepackage{rotating}
\usepackage{ytableau}
\usetikzlibrary{patterns}
\usetikzlibrary{decorations.pathreplacing}
\usetikzlibrary{calc,through}
\usetikzlibrary{backgrounds}

\allowdisplaybreaks[1]
\numberwithin{figure}{section}
\numberwithin{equation}{section}

\title{Major Index on Catalan combinatorics}
\author[K.~Shigechi]{Keiichi~Shigechi}
\email{k1.shigechi AT gmail.com}
\date{\today}

\newcommand\tikzpic[2]{
\raisebox{#1\totalheight}{
\begin{tikzpicture}
#2
\end{tikzpicture}
}}

\newtheorem{theorem}[figure]{Theorem}
\newtheorem{example}[figure]{Example}
\newtheorem{lemma}[figure]{Lemma}
\newtheorem{defn}[figure]{Definition}
\newtheorem{prop}[figure]{Proposition}
\newtheorem{cor}[figure]{Corollary}

\newtheorem{remark}[figure]{Remark}
\begin{document}

\begin{abstract}
We study the two statistics, the inversion number and the major index, on 
Catalan combinatorial objects such as  $r$-Dyck paths, $r$-Stirling permutations,
non-crossing partitions, Dyck tilings, and symmetric Dyck paths.
We show that they are equidistributed with the inversion number or the major 
index on Dyck paths.
\end{abstract}

\maketitle

\section{Introduction}
The Catalan numbers $\genfrac{}{}{}{}{1}{n+1}\genfrac{(}{)}{0pt}{}{2n}{n}=1,2,5,14,42,\ldots$ 
are one of the most well-known combinatorial numbers.
They count the numbers of combinatorial objects such as triangulation of a polygon, 
Dyck paths, binary trees, rooted planar trees, non-crossing partitions and so on (see e.g. 
\cite{Sta99,Sta15}).
Another remarkable combinatorial object is a permutation. 
A subset of permutations has been studied to study the set of Catalan objects. 
For example, the permutations which avoid a $w$-pattern $w\in\{132,231,312,321\}$ are known 
to be bijective to Dyck paths (see \cite{BanKil01,Knu73,Kra01,Stu09} and references therein). 
In this paper, we focus on the set of combinatorial objects which are bijective to 
the Catalan objects. They are $r$-Dyck paths with $r\ge1$, $r$-Stirling permutations, 
non-crossing partitions, Dyck tilings corresponding to $231$-avoiding permutations, and symmetric Dyck paths.

A statistic on combinatorial objects is an important tool to extract a combinatorial structure
from the objects.
The inversion number of a permutation is one of the most well-known statistic 
on a permutation, which goes back to at least \cite{DavBar62,Net01}.
Two different statistics are said to be equidistributed if they have the same 
generating function on the comibnatorial objects. 
A static is called Mahonian if it is equidistributed with the inversion statistic
on the permutations.
One of Mahonian statistics is the major index which is introduced by 
P.~A.~MacMahon in \cite{MacM16,MacM60}.
In this paper, we consider the inversion number and the major index on the combinatorial objects mentioned above.

The number of Dyck paths of size $n$ is given by the $n$-th Catalan number.
Since the number of permutations in $[n]:=\{1,2,\ldots,n\}$ which avoid the pattern $231$ is also given 
by the Catalan number, we have a bijection between Dyck paths and $231$-avoiding 
permutations \cite{Stu09}.
We make use of this property of the correspondence when we consider combinatorial objects 
such as $r$-Stirling permutations, non-crossing partitions, and Dyck tilings.

An $r$-Dyck path is a generalization of a Dyck path, and the number of $r$-Dyck paths 
is given by the Fuss--Catalan numbers. 
An $r$-Stirling permutation is a generalization of a permutation introduced by 
\cite{GesSta78,Par94}. As in the case of Dyck paths and permutations, we have 
a bijection between $r$-Dyck paths and $231$-avoiding $r$-Stirling permutations.
This bijection allows us to study the inversion number and major index on $r$-Stirling 
permutations instead of $r$-Dyck paths.
We introduce a new major index on $r$-Stirling permutations, and show that 
it is Mahonian (Theorem \ref{thrm:INVMAJ}). 

A non-crossing partition of $[n]$ is a combinatorial object whose total number 
is given by the Catalan number.
G. Kreweras initiated a systematical study of non-crossing partitions in \cite{Kre72}.
We define two statistics on a non-crossing partition which give the major index 
and the inversion number of a Dyck path (Theorem \ref{thrm:NCtoCat} and Proposition \ref{prop:NCqCat}
respectively).
By a refinement of the major index on a non-crossing partition, we obtain  
a generating function $\mathcal{A}(q,t)$ with two variables $q$ and $t$ which is reduced to 
the generating function of the major index at $t=q$ (Corollary \ref{cor:Aqq}). 
This polynomial $\mathcal{A}(q,t)$ is essentially the same as the polynomial $\mathcal{A}(q,t^{-1})$
studied in \cite{Stu09}.
This generalization of the major index is achieved by constructing a bijection between 
non-crossing permutations and $312$-avoiding permutations (equivalently $231$-avoiding 
permutations) (Theorem \ref{thrm:NCqt}).

A Dyck tiling is a tiling of the region between two Dyck paths by tiles which are called 
Dyck tiles. A Dyck tiling first appeared in the study of the Kazhdan--Lusztig polynomials for 
Grassmannian permutations in \cite{SZJ12}, and independently in the study of the double-dimer
models in \cite{KW11}.
In this paper, we consider a variant of Dyck tilings which is introduced in \cite{S19}.
Since we have a one-to-one correspondence between Dyck tilings and permutations, to relate 
them with Dyck paths, we consider the $231$-avoiding permutations.
By giving a weight on a Dyck tile, one can define the weight of a Dyck tiling which is 
equidistributed with the major index on the corresponding Dyck path (Theorem \ref{thrm:majD}).
The bijection between $231$-avoiding permutations and Dyck paths introduced in \cite{Stu09}
plays a central role to show that the weight of a Dyck tiling is equal to the major 
index of the corresponding Dyck path.
Similarly, by modifying the weight of a Dyck tile, we obtain a Mahonian statistic on 
Dyck tilings (Proposition \ref{prop:DTqCat}).

Finally, we consider the set of symmetric Dyck paths which are bijective to Dyck paths 
of size $n$.
We make use of the bijection between the two sets introduced in \cite[Section 7.3]{S23}.
We regard a path as a word of two alphabets $0$ and $1$, and this bijection does 
not preserve the length of a word in general. 
We define two statistics on a symmetric Dyck path which are equidistributed with the major 
index and the inversion number of a Dyck path.
A Dyck path of size $n$ has $n$ pairs of $0$ and $1$ in its word representation.
On the other hand, a symmetric Dyck path of size $n$ has $n-k$ pairs of $0$ and $1$, and 
$k$ unpaired $0$'s with $0\le k\le n$.
The existence of $k$ unpaired $0$'s allows us to refine the major index (more precisely 
the $q$-Narayana numbers) on symmetric Dyck paths (Theorem \ref{thrm:N3nrk}). 
Similarly, we newly define the weight on symmetric Dyck paths which is Mahonian 
(Proposition \ref{prop:symDPqCat}).

This paper is organized as follows.
In Section \ref{sec:CO}, we introduce the notion of $r$-Dyck paths, $r$-Stirling permutations,
and their pattern avoidance. 
In Section \ref{sec:Stat}, we introduce and study the inversion number and the major index on permutations, $r$-Dyck paths,
and $r$-Stirling permutations.
The inversion number and major index on non-crossing partitions are defined and studied in Section \ref{sec:NC}.
In Section \ref{sec:DT}, we study the Dyck tilings in the view of the inversion number and the major index on them.
In Section \ref{sec:SymDP}, the major index and inversion number on symmetric Dyck paths are shown 
to be equidistributed with the two statistics on Dyck paths respectively.

\paragraph{Notation.}
We will use the standard $q$-notation:
\begin{align*}
&(x;q)_{n}=(x)_{n}=(1-x)(1-qx)\cdots(1-q^{n-1}x)
=\sum_{k=0}^{n}q^{\genfrac{(}{)}{0pt}{}{k}{2}}\genfrac{[}{]}{0pt}{}{n}{k}(-x)^{k}, \\
&[n]=\sum_{i=0}^{n-1}q^{i},\qquad \genfrac{[}{]}{0pt}{}{n}{k}=\genfrac{}{}{}{}{(q)_n}{(q)_k(q)_{n-k}}.
\end{align*}
The Kronecker delta function $\delta(P)$ is defined by 
\begin{align*}
\delta(P):=\begin{cases}
1, & \text{ if $P$ is true}, \\
0, & \text{ otherwise}. 
\end{cases}
\end{align*}

\section{Combinatorial objects}
\label{sec:CO}
\subsection{Generalized Dyck paths}
A {\it generalized Dyck path} of length $n$ or {\it $r$-Dyck path} with $r\ge1$ is 
an up-right lattice path from $(0,0)$ to $(rn,n)$ such that 
the path never goes down the line $y=x/r$.
Since an $r$-Dyck path $p$ consists of up steps $(0,1)$ and right steps $(1,0)$, 
we express $p$ in terms of $0$ and $1$ where $0$ (resp. $1$) stands for an up (resp. right)
step.
For example, we have five Dyck paths of length three:
\begin{align*}
000111 \qquad 001011 \qquad 001101 \qquad 010011 \qquad 010101
\end{align*}
   
A Dyck path $d$ is said to be {\it prime} if $d$ touches the line $y=x/r$ exactly twice.
For example, we have two prime Dyck paths of size three: $000111$ and $001011$.

We denote by $\mathcal{P}^{r}_{n}$ be the set of $r$-Dyck paths of length $n$.
The set $\mathcal{P}^{r}_{n}$ has two distinguished paths: 
the lowest path $(01^{r})^{n}$ and the top path $0^{n}1^{rn}$.

The number $|\mathcal{P}^{r}_{n}|$ of $r$-Dyck paths of length $n$ is given 
by the Fuss--Catalan number
\begin{align*}
|\mathcal{P}^{r}_{n}|=\genfrac{}{}{}{}{1}{nr+1}\genfrac{(}{)}{0pt}{}{(r+1)n}{n}.
\end{align*}
When $r=1$, the number of Dyck paths of size $n$ is given by the well-known 
Catalan numbers.

\subsection{\texorpdfstring{$r$}{r}-Stirling permutations}
When $r=1$, the set $\mathcal{P}^{1}_{n}$ of Dyck paths is bijective to a subset of 
the permutations of $n$ letters. A Dyck path may be characterized by 
the pattern avoidance of the permutations. 
For example, a region between the top Dyck path and a Dyck path gives a Young diagram.
Then, we have a one-to-one correspondence between Dyck paths and 
$132$-avoiding permutations.
By considering a correspondence between Dyck paths and other combinatorial 
objects whose total number is given by the Catalan number, one has 
a correspondence between Dyck paths and $w$-avoiding permutations with 
$w\in\{132,231,312,321\}$ (see e.g. \cite{BanKil01,Knu73,Kra01}).
We denote the set of $w$-avoiding permutations by $\mathfrak{S}(w)$.

For $r\ge2$, $r$-Stirling permutations play a role instead of permutations 
for $r=1$.
The $r$-Stirling permutations were introduced by I.~Gessel and R.~P.~Stanley \cite{GesSta78} for $r=2$,
and generalized to the case $r\ge3$ by S.~Park \cite{Par94}.

An {\it $r$-Stirling permutation of size $n$} is a permutation of the multiset 
$\{1^{r},2^r,\ldots,n^r\}$ such that 
if an integer $i$ appears between two $j$'s then $i<j$.
Let $\mathfrak{S}^{(r)}_{n}$ denote the set of $r$-Stirling permutations 
of size $n$.
When $r=1$, we have the set $\mathfrak{S}_{n}:=\mathfrak{S}^{(1)}_{n}$ of the permutations
of $n$ letters.
The total number of the set $\mathfrak{S}^{(r)}_{n}$ is given by
\begin{align*}
|\mathfrak{S}^{(r)}_{n}|=\prod_{s=1}^{n-1}(sr+1).
\end{align*} 

For example, we have fifteen $2$-Stirling permutations of size three:
\begin{align*}
332211 \qquad 332112 \qquad 322311 \qquad 331122 \qquad 322113 \\
223311 \qquad 311322 \qquad 321123 \qquad 223113 \qquad 113322 \\
311223 \qquad 221133 \qquad 113223 \qquad 211233 \qquad 112233
\end{align*} 

\subsection{Bijection between \texorpdfstring{$r$}{r}-Dyck paths and 
\texorpdfstring{$231$}{231}-avoiding \texorpdfstring{$r$}{r}-Stirling permutations}
\label{sec:bij231}
An $r$-Stirling permutation $\pi:=\pi_1\ldots\pi_{nr}\in\mathfrak{S}^{(r)}_{n}$ 
is said to be $231$-avoiding if there is no integers $i<j<k$ such that 
$\pi_{k}<\pi_{i}<\pi_{j}$.
For $2$-Stirling permutations of size three, we have three $231$-containing $2$-Stirling 
permutations: $322311$, $223311$, and $223113$.

\begin{defn}
We denote by $\mathfrak{S}^{(r)}_{n}(231)$ the set of $231$-avoiding $r$-Stirling permutations.
\end{defn}

In what follows, we give a bijection between $r$-Dyck paths of size $n$ and 
$231$-avoiding $r$-Stirling permutations of size $n$.

Let $\pi:=\pi_1\ldots\pi_{rn}$ be a $231$-avoiding $r$-Stirling permutation.
We define a sequence $\alpha(\pi):=(\alpha_1,\ldots,\alpha_{n})$ 
of non-negative integers of length $n$ from $\pi$ 
by 
\begin{align*}
\alpha_{i}:=\#\{j>i: \text{ $j$ is left to the left-most $i$ in $\pi$} \}.
\end{align*}
By definition, we have $\alpha_{n}=0$.
For example, we have $\alpha(\pi)=(2,1,1,0)$ if $\pi=42112334\in\mathfrak{S}^{(2)}_{4}$.

Let $p_{0}$ be the lowest $r$-Dyck path of length $n$, {\it i.e.}, $p_{0}=(01^{r})^{n}$.
We construct an $r$-Dyck path $p(\pi)$ for $\pi$ from $\alpha(\pi)$.
We add $\alpha_{i}$ unit boxes in the $i$-th row of $p_{0}$ from top to bottom.
Then, we define $p(\pi)$ as the top path.
For example, $p(42112334)$ is the $2$-Dyck path $010110101111$.
Graphically, we have 
\begin{align*}
\tikzpic{-0.5}{[scale=0.5]
\draw[thick](0,0)--(0,1)--(2,1)--(2,2)--(4,2)--(4,3)--(6,3)--(6,4)--(8,4);
\draw(1,1)--(1,2)--(2,2)(3,2)--(3,3)--(4,3)--(4,4)--(6,4)(5,4)--(5,3);
}
\end{align*}
Note that we have two, one and one boxes in the first, second and third rows.

We denote the map $\kappa:\mathfrak{S}^{(r)}_{n}(231)\rightarrow\mathcal{P}^{r}_{n}$ defined 
as above.

\begin{prop}
The map $\kappa$ is well-defined, that is, $\kappa(\pi)\in\mathcal{P}^{r}_{n}$ 
if $\pi\in\mathfrak{S}^{(r)}_{n}(231)$.
Further, $\kappa$ is a bijection between the two sets.
\end{prop}
\begin{proof}
We first show that $\kappa$ is well-defined.
Let $\pi\in\mathfrak{S}^{(r)}_{n}(231)$ be an $r$-Stirling permutation. Then, by the definition 
of $\alpha(\pi)$, we always have $\alpha_{i}<\alpha_{i+1}+r+1$, since $\pi$ avoids a pattern 
$231$ and $\pi\in\mathfrak{S}^{(r)}_{n}$.
On the other hand, recall that $p_{0}$ is the lowest $r$-Dyck path in $\mathcal{P}^{r}_{n}$.
This implies that if we put $\beta_{i}$ unit boxes in the $i$-th row above the lowest path $p_{0}$, 
the numbers $\beta_{i}$ satisfy the relation $\beta_{i}<\beta_{i+1}+r+1$ to have an $r$-Dyck path.
Then, by construction of $\kappa$, we have $\alpha_{i}=\beta_{i}$, and $\kappa(\pi)$ is 
an $r$-Dyck path. Therefore, the map $\kappa$ is well-defined.

By construction of $\kappa$, the inverse map $\kappa^{-1}$ is easily obtained.
Let $p$ be an $r$-Dyck path. Then, we count the numbers $\alpha_{i}$ of unit boxes in the $i$-th row 
which are below $p$ and above the lowest path $p_{0}$. Once we have $\alpha$, it is obvious that we have 
a $231$-avoiding $r$-Stirling permutation.
These imply that $\kappa^{-1}$ is indeed an inverse of $\kappa$.
Thus, it is a bijection between the two sets $\mathcal{P}^{r}_{n}$ and $\mathfrak{S}^{(r)}_{n}(231)$.
\end{proof}

\section{Statistics}
\label{sec:Stat}
Let $f:=(x_1,\ldots,x_{n})\in\mathbb{R}^{n}$ be an arbitrary sequence of 
real numbers of length $n$.
The {\it inversion number} $\mathrm{inv}(f)$ of $f$ is defined 
as the number of pairs $(j,k)$ such that 
$1\le j<k\le n$ and $x_{j}>x_{k}$.
The {\it major index} $\mathrm{maj}(f)$ of $f$ is defined as 
the sum of all integers $j$ such that $1\le j\le n-1$ and $x_{j}>x_{j+1}$.
Given a word $f$, we denote the set of descents by $\mathrm{Des}(f)$, {\it i.e.},
$\mathrm{Des}(f):=\{i : f_{i}>f_{i+1}\}$. 

\subsection{Dyck paths and permutations}
Given a permutation $\omega$, the inversion number of $\omega$ 
is defined to be $\mathrm{inv}(\omega)$.
Similarly, the major index of $\omega$ is given by $\mathrm{maj}(\omega)$.

\begin{prop}
We have 
\begin{align*}
\sum_{\omega\in\mathfrak{S}_{n}}q^{\mathrm{inv}(\omega)}=\sum_{\omega\in\mathfrak{S}_{n}}q^{\mathrm{maj}(\omega)}=[n]!.
\end{align*}
\end{prop}
\begin{proof}[Sketch of the proof]
We prove the statement by induction on $n$.
We decompose the permutation $\omega$ as $\omega=\omega_11\omega_2$.
Then, we have 
$\mathrm{inv}(\omega)=\mathrm{inv}(\omega_1\omega_2)+|w_1|$. 
This implies 
\begin{align*}
\sum_{\omega\in\mathfrak{S}_{n}}q^{\mathrm{inv}(\omega)}
=\sum_{j=0}^{n-1}\sum_{\omega'\in\mathfrak{S}_{n-1}}q^{\mathrm{inv}(\omega')+j}
=[n]\cdot [n-1]!=[n]!.
\end{align*}
One can prove the statement for the major index by a similar way.
\end{proof}

Following \cite{Car72,CarRio64,FurHof85}, we define $q$-Catalan numbers $C_{n}=C_{n}(q)$ by
\begin{align}
\label{eq:defCn}
z=\sum_{n=1}^{\infty}C_{n-1}z^{n}(1-z)(1-qz)\cdots(1-q^{n-1}z).
\end{align}
The first few values are 
\begin{align*}
&C_0=C_1=1, \qquad C_{2}=1+q, \qquad C_{3}=1+q+2q^2+q^3, \\
&C_{4}=1+q+2q^2+3q^3+3q^4+3q^5+q^6.
\end{align*}
The $q$-Catalan number $C_{n}$ satisfies the recurrence relation (see e.g. \cite{FurHof85})
\begin{align}
\label{eq:recCn}
C_{n+1}=\sum_{k=0}^{n}C_{k}C_{n-k}q^{(k+1)(n-k)},
\end{align}
with $C_0=1$. Then, the exponent of $q$ in $C_{n}$ counts the area between the Dyck paths 
and the top path $0^{n}1^{n}$, {\it i.e.}, 
\begin{align}
\label{eq:AreaCn}
\sum_{d\in\mathcal{P}^1_n}q^{\mathrm{Area}(d)}=C_{n},
\end{align} 
where $\mathrm{Area}(d)$ is the number of unit boxes above $d$ and below $0^n1^n$.

Recall that $\mathfrak{S}_{n}(132)$ is the set of $132$-avoiding 
permutations of size $n$. 
\begin{prop}
We have 
\begin{align}
\label{eq:132C}
\sum_{\omega\in\mathfrak{S}_{n}(132)}q^{\mathrm{inv}(\omega)}
=C_{n}(q).
\end{align}
\end{prop}
\begin{proof}
Let $\mathbf{h}(\omega):=h_{1}h_{2}\ldots h_{n}$, $1\le i\le n$, be a sequence of non-negative integers 
defined by
\begin{align*}
h_{i}:=\#\{j>i : \text{ $j$ is left to $i$ in $\omega$ }\}.
\end{align*}
It is easy to see that $0\le h_{i}\le i-1$.
Further, if $\omega\in\mathfrak{S}_{n}(132)$, $\mathbf{h}(\omega)$ satisfies 
$h_{i}\ge h_{i+1}$.
Then, $\mathbf{h}(\omega)$ is bijective to a Young diagram $\lambda$ inside of the 
staircase Young diagram $(n-1,n-2,\ldots,1,0)$.
The sum $\sum_{i}h_{i}$ is equal to the number of inversions in $\omega$.  
By the identification of $\mathbf{h}(\omega)$ with $\lambda$, 
this sum is equal to the number of unit boxes in the Young diagram $\lambda$.
From these, we have Eq. (\ref{eq:132C}).
\end{proof}

Following \cite{FurHof85}, we define another $q$-Catalan numbers $E_{n}=E_{n}(q)$ 
by the expansion formula
\begin{align*}
z=\sum_{n=1}^{\infty}
\genfrac{}{}{}{}{E_{n}(q)z^{n}}{q^{\genfrac{(}{)}{0pt}{}{n}{2}}(-q^{-n}z)_{n}(-qz)_{n}}.
\end{align*}
By using the $q$-Lagrange formula, one can obtain  
\begin{align*}
E_{n}(q)=\genfrac{}{}{}{}{1}{[n+1]}\genfrac{[}{]}{0pt}{}{2n}{n}.
\end{align*}

Let $d$ be a Dyck path, i.e., a word of $0$'s and $1$'s. 
Then, the major index of $d$ is given by $\mathrm{maj}(d)$.

\begin{prop}[\cite{FurHof85}]
We have 
\begin{align*}
\sum_{d\in\mathcal{P}^{1}_{n}}q^{\mathrm{maj}(d)}
=E_{n}(q)=\genfrac{}{}{}{}{1}{[n+1]}\genfrac{[}{]}{0pt}{}{2n}{n}.
\end{align*}
\end{prop}
		
\subsection{\texorpdfstring{$r$}{r}-Dyck paths and \texorpdfstring{$r$}{r}-Stirling permutations}
Recall that an $r$-Dyck path $\mu\in\mathcal{P}^{r}_{n}$ is an up-right path 
starting from $(0,0)$. Since the path $\mu$ is below the top path, the region 
between $\mu$ and the top path gives a Young diagram $Y(\mu)$.
The inversion numbers $\mathrm{Inv}(\mu)$ and $\mathrm{iInv}(\mu)$ are defined by 
\begin{align*}
\mathrm{Inv}(\mu)&:=\genfrac{}{}{}{}{r}{2}n(n-1)-|Y(\mu)|, \\
\mathrm{iInv}(\mu)&:=|Y(\mu)|,
\end{align*}
where $|Y(\mu)|$ is the number of unit squares in the Young diagram $Y(\mu)$.

We generalize the results of \cite{Car72,CarRio64,FurHof85} to $r$-Dyck paths.
Let $C_n^{(r)}$ be the number of $r$-Dyck paths of size $n$.
The generating function $f(t):=\sum_{n=0}^{\infty}C_{n}^{(r)}t^{n}$ of $r$-Dyck paths 
satisfies the recurrence relation
\begin{align*}
f(t)=1+tf(t)^{r+1}.
\end{align*}
We set $z=t^{1/r}f(t)$. 
Then, $z$ satisfies 
\begin{align*}
z=\sum_{n=0}^{\infty}C_{n}^{(r)}z^{rn+1}(1-z^r)^{rn+1}.
\end{align*}
We define the $q$-analogue of Fuss--Catalan numbers $C^{(r)}_{n}:=C^{(r)}_n(q)$ by 
\begin{align}
\label{eq:Cnrinsum}
z=\sum_{n=0}^{\infty}C_{n}^{(r)}(q)z^{rn+1}(1-z^r)(1-qz^r)(1-q^2z^r)\cdots(1-q^{rn}z^r).
\end{align}
As in the case of $r=1$, the exponent of $q$ in $C_{n}^{(r)}$ counts the 
area between the $r$-Dyck paths and the top path $0^n1^{rn}$.

\begin{prop}
We have 
\begin{align}
\label{eq:iInvCrn}
\sum_{\mu\in\mathcal{P}^{r}_n}q^{\mathrm{iInv}(\mu)}=C^{(r)}_{n}(q).
\end{align}
\end{prop}
\begin{proof}
Let $C'_{n}(q)$ be the sum $\sum_{\mu\in\mathcal{P}^{r}_{n}}q^{\mathrm{Inv}(\mu)}$. 
Then, the left hand side of Eq. (\ref{eq:iInvCrn}) is equal to 
$\widetilde{C}_{n}(q):=q^{r\genfrac{(}{)}{0pt}{}{n}{2}}C'_{n}(q^{-1})$.
We decompose an $r$-Dyck path $d$
as $d=0d_11d_21\cdots d_{r}1d_{r+1}$, where $d_{i}\in\mathcal{P}^{r}_{n_i}$ with $n_i\ge0$.
The generating function $f'(t):=\sum_{n=0}^{\infty}C'_{n}(q)t^{n}$ satisfies the recurrence 
relation 
\begin{align*}
f'(t)=1+tf'(t)f'(qt)f'(q^2t)\cdots f'(q^{r}t).
\end{align*}
The sum $C'_{n}(q)$ satisfies the recurrence relation
\begin{align}
\label{eq:Cdashed}
C'_{n}(q)=\sum_{l_0+\ldots+l_{r}=n-1}C'_{l_{0}}(q)\ldots C'_{l_{r}}(q)q^{\sum_{j=1}^{r}jl_j}.
\end{align}
Then, the left hand side of Eq. (\ref{eq:iInvCrn}) satisfies the recurrence relation
\begin{align*}
\widetilde{C}_{n}(q)=\sum_{l_0+\ldots+l_r=n+1}
\widetilde{C}_{l_0}(q)\ldots \widetilde{C}_{l_r}(q)q^{\sum_{j=1}^{r}l_{j}(j+r\sum_{i=0}^{j-1}l_i)}.
\end{align*}
From Eq. (\ref{eq:Cnrinsum}), we have 
\begin{align*}
z^{r+1}&=\sum_{k=0}^{\infty}C_{k}^{(r)}(q)z^{rk+1}(z^r)_{rk+1}
\prod_{j=1}^{r}q^{-k-j/r}\cdot (q^{k+j/r}z), \\
&=\sum_{k=0}^{\infty}C_{k}^{(r)}(q)z^{rk+1}(z^r)_{rk+1}
\prod_{j=1}^{r}q^{-k-j/r}\sum_{l_j=0}^{\infty}C_{l_j}(q)(q^{k+j/r}z)^{rl_j+1}(q^{rk+j}z^r)_{rl_{j}+1},\\
&=\sum_{n=0}^{\infty}
\left(\sum_{k+l_1+\ldots+l_r=n-1}C_{k}C_{l_1}\ldots C_{l_r}q^{\sum_{j=1}^{r}(j+r\sum_{i=1}^{r}j_i)l_j} \right)z^{rn+1}(z^r)_{rn+1}.
\end{align*}
We rewrite Eq. (\ref{eq:Cnrinsum}) as 
\begin{align*}
z&=C_{0}z(1-z^r)+\sum_{k=1}^{\infty}C_{k}z^{rk+1}(z^r)_{rk+1}, \\
&=z-z^{r+1}+\sum_{k=1}^{\infty}C_{k}z^{rk+1}(z^r)_{rk+1}.
\end{align*}
We have two expressions for $z^{r+1}$.
By comparing coefficients of the two expressions above, we obtain the  recurrence 
relation for $C_{n}^{(r)}(q)$, which is nothing but Eq. (\ref{eq:Cdashed}). 
From these, we have Eq. (\ref{eq:iInvCrn}), which completes the proof.
\end{proof}

Let $d\in\{0,1\}^{\ast}$ be an $r$-Dyck path in $\mathcal{P}^{r}_{n}$.
The major index of $d$ is given by $\mathrm{maj}(d)$.

We define a new kind of $q$-analogue of the Fuss--Catalan numbers $E^{(r)}_{n}=E^{(r)}_{n}(q)$
by means of the expansion formula
\begin{align}
\label{eq:Einqq}
z=\sum_{n=1}^{\infty}
\genfrac{}{}{}{}{E^{(r)}_{n}z^{n}}{q^{\genfrac{(}{)}{0pt}{}{n}{2}}(-q^{-n}z)_{n}(-qz)_{rn}}.
\end{align}

\begin{prop}
\label{prop:majE}
We have
\begin{align}
\label{eq:majEinqq}
\sum_{d\in\mathcal{P}^{r}_{n}}q^{\mathrm{maj}(d)}=E^{(r)}_{n}(q).
\end{align}
\end{prop}

To prove Proposition \ref{prop:majE}, we consider a generalization of Eq. (\ref{eq:Einqq}).
We refine the major statistic following \cite[Section 5]{FurHof85}:
For a word $w:=w_1w_2\ldots w_{n}\in\{0,1\}^{\ast}$, we set
\begin{align*}
\alpha(w)&:=\sum_{i\in\mathrm{Des}(w)}|\{j\le i : w_{j}=0\}|, \\
\beta(w)&:=\sum_{i\in\mathrm{Des}(w)}|\{j\le i : w_{j}=1\}|,
\end{align*}
where $\mathrm{Des}(w)$ is the descent set of $w$. 
It is obvious that $\mathrm{maj}(w)=\alpha(w)+\beta(w)$.

\begin{theorem}
\label{thrm:E}
The following statements are equivalent:
\begin{align}
\label{eq:defE}
&E^{(r)}_{n}(x;a,b):=\sum_{d\in\mathcal{P}^{r}_{n}}
x^{|\mathrm{Des}(d)|}a^{\alpha(d)}b^{\beta(d)}, \\
\label{eq:Einprod}
&z=\sum_{n=1}^{\infty}\genfrac{}{}{}{}{a^{-\genfrac{(}{)}{0pt}{}{n}{2}}E_{n}^{(r)}(x;a,b)z^{n}}
{(1+a^{-1}z)\cdots(1+a^{-n}z)(1+xbz)\cdots(1+xb^{rn}z)}, \\
\label{eq:Erec}
&E_{n+1}^{(r)}(x)=\sum_{n_0+n_1+\ldots+n_{r}=n}
E_{n_0}^{(r)}(ax)\prod_{i=1}^{r}E_{n_i}(a^{d_1(i)}b^{d_2(i)}x)
(a^{d_{1}(i)}b^{d_{2}(i)}x)^{\delta(n_i>0)},
\end{align}
where $E_{0}^{(r)}(x):=1$, $\delta(P)$ is the Kronecker delta function and 
\begin{align*}
d_1(i):=1+\sum_{j=1}^{i-1}n_{j}, \qquad d_2(i):=i+\sum_{j=1}^{i-1}n_{j}.
\end{align*}
\end{theorem}
\begin{proof}
We first show (\ref{eq:defE}) $\Leftrightarrow$ (\ref{eq:Erec}).
We decompose the word $w\in\mathcal{P}^{r}_{n}$ into 
$w=0w_01w_1\ldots1w_{r}$ with $w_{i}\in\mathcal{P}^{r}_{n_{i}}$ for $0\le n_{i}\le n$.
Then, we have 
\begin{align*}
|\mathrm{Des}(w)|&=\sum_{i=0}^{r}|\mathrm{Des}(w_{i})|+\sum_{i=1}^{r}\delta(n_{i}>0),  \\
\alpha(w)&=\sum_{i=0}^{r}\alpha(w_{i})+\sum_{i=1}^{r}d_1(i)(1+|\mathrm{Des}(w_i)|)\delta(n_{i}>0),\\
\beta(w)&=\sum_{i=0}^{r}\alpha(w_{i})+\sum_{i=1}^{r}d_2(i)(1+|\mathrm{Des}(w_i)|)\delta(n_{i}>0).\\
\end{align*}
This gives the recurrence relation (\ref{eq:Erec}).

(\ref{eq:Einprod}) $\Leftrightarrow$ (\ref{eq:Erec}). 
We use the same method as in \cite[Section 5]{FurHof85}. 
We define $F_{n}^{(i)}(x)$ by the following expression: 
\begin{align}
\label{eq:defF}
z=\sum_{n=0}^{\infty}
\genfrac{}{}{}{}{a^{-\genfrac{(}{)}{0pt}{}{n}{2}}F_{n}^{(i)}(x)z^{n+1}}{(-z;a^{-1})_{n+1}(-xbz;b)_{rn+i}}.
\end{align}
We add the factor $(1+xb^{rn+i+1}z)$ in the denominator of the right hand side of Eq. (\ref{eq:defF}).
Note that $(1+xb^{rn+i+1}z)=(1+xa^{n}b^{rn+i+1}(a^{-n}z))$. We replace $a^{-n}z$ by the series (\ref{eq:Einprod})
with $xa^{n}b^{rn+i+1}$ instead of $x$.
Then, Eq. (\ref{eq:defF}) can be written as 
\begin{align*}
z&=\sum_{n=0}^{\infty}
\genfrac{}{}{}{}{a^{-\genfrac{(}{)}{0pt}{}{n}{2}}F_{n}^{(i)}(x)z^{n+1}}{(-z;a^{-1})_{n+1}(-xbz;b)_{rn+i+1}}
\left(1+xa^{n}b^{rn+i+1}\sum_{j=1}^{\infty}\genfrac{}{}{}{}{a^{-\genfrac{(}{)}{0pt}{}{j}{2}}E_{n}(xa^{n}b^{rn+i+1})z^{j}a^{-nj}}
{(-za^{-n-1};a^{-1})_{n+1}(-xb^{rn+i+2}z;b)_{rj}}
\right), \\
&=\sum_{n=0}^{\infty}\genfrac{}{}{}{}{a^{-\genfrac{(}{)}{0pt}{}{n}{2}}z^{n+1}}{(-z;a^{-1})_{n+1}(-xbz;b)_{rn+i+1}}
\left(F_{n}^{(i)}(x)+x\sum_{k=0}^{n-1}a^{k}b^{rk+i+1}F_{k}^{(i)}E_{n-k}(xa^kb^{rk+i+1})\right),
\end{align*}
where $E_{n}(x):=E_{n}^{(r)}(x;a,b)$ for short.
By definition of $F_{n}^{(i)}(x)$, we have 
\begin{align}
\label{eq:Frec}
F_{n}^{(i+1)}(x)=F_{n}^{(i)}(x)+x\sum_{k=0}^{n-1}a^{k}b^{rk+i+1}F_{k}^{(i)}E_{n-k}(xa^kb^{rk+i+1}).
\end{align}
Since we have $F_{n}^{(0)}(x)=E_{n}(x)$, and $F_{n}^{(r)}$ is also expressed in terms of $E_{n}(x)$,
we obtain the recurrence relation (\ref{eq:Erec}).
\end{proof}

\begin{proof}[Proof of Proposition \ref{prop:majE}]
From Theorem \ref{thrm:E}, Eq. (\ref{eq:Einqq}) is obtained by setting 
$(a,b,x)=(q,q,1)$.
Then, we have Eq. (\ref{eq:majEinqq}) from Eqs. (\ref{eq:defE}) and (\ref{eq:Einprod}).
\end{proof}

For example, the first few values of $E_{n}^{(2)}$ are given by 
\begin{align*}
E_{1}^{(2)}&=1, \\
E_{2}^{(2)}&=1+abx+ab^2x, \\
E_{3}^{(2)}&=1+abx+a^2bx+ab^2x+a^2b^2x+a^2b^3x+a^2b^4x+
a^3b^3x^2+a^3b^4x^2+2a^3b^5x^2+a^3b^6x^2.
\end{align*}

\subsection{Major index on \texorpdfstring{$r$}{r}-Stirling permutations}
We define two statistics, which we call inversion number and major index, 
on the $r$-Stirling permutations.
Let $\nu\in\mathfrak{S}^{(r)}_{n}$ be an $r$-Stirling permutation of size $n$. 

\begin{defn}
\label{defn:INVMAJ}
The inversion number $\mathrm{INV}(\nu)$ of $\nu$ is given by 
\begin{align*}
\mathrm{INV}(\nu):=\genfrac{}{}{}{}{1}{r}\mathrm{inv}(\nu). 
\end{align*}
Similarly, the major index $\mathrm{MAJ}(\nu)$ of $\nu$ is define by 
\begin{align*}
\mathrm{MAJ}(\nu):=\sum_{j\in J(\nu)}j,
\end{align*}
where $J(\nu)$ is the set of positive integers 
\begin{align}
\label{eq:Jnu}
J(\nu)=\{1\le j\le r(n-1): \nu_{j}>\nu_{j+1}=\nu_{j+2}=\cdots=\nu_{j+r}\}.
\end{align}
\end{defn}

\begin{remark}
In the case of $r=1$, the statistics $\mathrm{INV}(\nu)$ and $\mathrm{MAJ}(\nu)$
are equivalent to $\mathrm{inv}(\nu)$ and $\mathrm{maj}(\nu)$ for 
$\nu\in\mathfrak{S}_{n}$.
\end{remark}

\begin{example}
Let $(n,r)=(3,2)$. Then, the inversion numbers and major indices for 
the $2$-Stirling permutations of size three are given in
Table \ref{table:majinv}.
\begin{table}[ht]
\begin{tabular}{c|cccccccc}\hline
 & $112233$ &  $211233$ & $221133$ & $113223$ & $311223$ & $321123$ & $322113$ & $113322$\\ \hline
$\mathrm{INV}$ & $0$ & $1$ & $2$ & $1$ & $2$ & $3$ & $4$ & $2$ \\ \hline
$\mathrm{MAJ}$ & $0$ & $1$ & $2$ & $3$ & $1$ & $2$ & $4$ & $4$ \\ \hline 
\end{tabular}
\\[12pt]
\begin{tabular}{c|cccc|ccc}\hline
  & $311322$ & $331122$ & $332112$ & $332211$ & $223113$ & $223311$ & $322311$  \\ \hline
$\mathrm{INV}$  & $3$ & $4$ & $5$ & $6$ & $3$ & $4$ & $5$ \\ \hline
$\mathrm{MAJ}$  & $5$ & $2$ & $3$ & $6$ & $3$ & $4$ & $5$ \\ \hline 
\end{tabular}\\[12pt]
\caption{The major indices and inversion numbers for $2$-Stirling permutations of size three.}
\label{table:majinv}
\end{table}
\end{example}
The last three $2$-Stirling permutations are not $231$-avoiding $2$-Stirling permutations.

\begin{theorem}
\label{thrm:INVMAJ}
We have 
\begin{align*}
\sum_{\nu\in\mathfrak{S}^{(r)}_{n}}q^{\mathrm{INV}(\nu)}
=\sum_{\nu\in\mathfrak{S}^{(r)}_{n}}q^{\mathrm{MAJ}(\nu)}
=[r+1][2r+1]\cdots[(n-1)r+1].
\end{align*}
\end{theorem}

To prove Theorem \ref{thrm:INVMAJ}, we will construct a permutation $\Phi$ satisfying 
\begin{align}
\label{def:Phi}
\mathrm{INV}(\Phi(\nu))=\mathrm{MAJ}(\nu) \qquad \text{for all } \nu\in\mathfrak{S}^{(r)}_{n},
\end{align}
and such that $\Phi(\nu)\in\mathfrak{S}^{(r)}_{n}$.
In the case of $r=1$, Foata gave an explicit construction of $\Phi$
in \cite{Foa68}. 
This motivates us to construct the permutation $\Phi$ in the case 
of general $r\ge2$.

In what follows, we give a map $\Phi:\mathfrak{S}_{n}^{(r)}\rightarrow\mathfrak{S}_{n}^{(r)}$ satisfying 
Eq. (\ref{def:Phi}) for all $\nu\in\mathfrak{S}^{(r)}_{n}$.
Let $\nu$ be an $r$-Stirling permutation in the set $[2,n]$.
To obtain an $r$-Stirling permutation in $[1,n]$, we insert $1^{r}$ somewhere in $\nu$.
Let $\nu_{i}$, $0\le i\le r(n-1)$, be an $r$-Stirling permutation obtained from $\nu$
by inserting $1^r$ from the $i$-th position in $\nu$. 
Define two numbers $P_{i}$ and $Q_{i}$ by
\begin{align*}
P_{i}&:=\mathrm{INV}(\nu_{i})-\mathrm{INV}(\nu), \\
Q_{i}&:=\mathrm{MAJ}(\nu_{i})-\mathrm{MAJ}(\nu).	
\end{align*} 

\begin{lemma}
\label{lemma:inspos}
We have 
\begin{align*}
\{P_{i} : 0\le i\le r(n-1)\}=\{Q_{i} : 0\le i\le r(n-1) \}=[0,r(n-1)].
\end{align*}
\end{lemma}
\begin{proof}
Since the integer $1$ is the smallest in the set $[1,n]$, it is obvious 
that we have $P_{i}=i$ for $0\le i\le r(n-1)$. Therefore, we have 
$\{P_{i}\}=[0,r(n-1)]$.

Suppose that $\{i_1,\ldots,i_{m}\}$ is the set of descents of $\nu$.
We first show that all $Q_{i}$ are distinct.
From Eq. (\ref{eq:Jnu}), we have 
\begin{align*}
Q_{i_k+s}=
\begin{cases}
r(m-k)+s, & 0\le s\le r-1, \\
r(m-k)+i_{k}+s, & r\le s\le i_{k+1}-i_{k}-1.
\end{cases}
\end{align*}
When $0\le s,s'\le r-1$, it is obvious that we have $Q_{i_j+s}\neq Q_{i_j'+s'}$ if $j\neq j'$.
Suppose we have $Q_{i_j+s}=Q_{i_j'+s'}$ with $r\le s,s'$ and $j<j'$.
The condition $Q_{i_j+s}=Q_{i_j'+s'}$ implies that 
\begin{align*}
i_{j'}-i_{j}&=r(j'-j)+s-s', \\
&\le r(j'-j)+i_{j+1}-i_{j}-1-r.
\end{align*}
The above inequality is equivalent to
\begin{align*}
i_{j'}-i_{j}\le r(j'-j)-1,
\end{align*}
for $j<j'$. However, the condition (\ref{eq:Jnu}) insures that 
we have $i_{j'}-i_{j}\ge r(j'-j)$, which gives a contradiction.
Therefore, we have $Q_{i_j+s}\neq Q_{i_j'+s'}$ with $r\le s,s'$ and $j<j'$.
Finally, suppose that $Q_{i_j+s}=Q_{i_j'+s'}$ with $0\le s\le r-1$ and 
$r\le s'$. The equation $Q_{i_j+s}=Q_{i_j'+s'}$ implies that 
\begin{align*}
i_{j'}&=r(j'-j)+s-s', \\
&\le r(j'-j)+r-1-r, \\
&=r(j'-j)-1.
\end{align*}
The condition (\ref{eq:Jnu}) again implies $i_{j'}\le r(j'-1)+1$. These two 
conditions imply that $r(j-1)\le -2$, which is a contradiction since we have 
$j\ge1$. Therefore, we have $Q_{i_j+s}\neq Q_{i_j'+s'}$ with $0\le s\le r-1$ and 
$r\le s'$.
From these, all $Q_{i}$ are distinct. 
It is easy to see that $Q_{i}\ge0$.
We prove that $Q_{i}\le r(n-1)$.
The condition (\ref{eq:Jnu}) implies that the number of descents in $\nu$ 
is at most $n-1$. 
If $Q_{i_k+s}\ge r(n-1)+1$ for $0\le s\le r-1$, we have 
\begin{align*}
r(n-1)+1&\le Q_{i_k+s}=r(m-k)+s , \\
&<r(m-k)+r.
\end{align*}
This inequality is equivalent to $n-1\le m-k$.
However, we have $m-k\le n-1-1=n-2$, which is a contradiction.
Therefore, we have $Q_{i}\le r(n-1)$.
Recall that $0\le i\le r(n-1)$, $0\le Q_{i}\le r(n-1)$, and 
all $Q_{i}$ are distinct. From these, we have 
$\{Q_{i}: 0\le i\le r(n-1)\}=[0,r(n-1)]$, which completes the proof.
\end{proof}

We define a map $\Phi$ recursively as follows.
Let $\nu'$ be an $r$-Stirling permutation of the set $[2,n]$, and 
$\Phi(\nu')$ be an $r$-Stirling permutation of the set $[2,n]$
satisfying the condition $\mathrm{INV}(\Phi(\nu'))=\mathrm{MAJ}(\nu')$.
Let $\nu$ be an $r$-Stirling permutation of the set $[1,n]$ obtained from 
$\nu'$ by inserting $1^r$.
We insert $1^r$ in $\Phi(\nu')$ and obtain $\widetilde{\nu}$. 
From Lemma \ref{lemma:inspos}, we have a unique position such that 
$\mathrm{MAJ}(\nu)=\mathrm{INV}(\widetilde{\nu})$.
Then, the map $\Phi$ on $\nu$ is defined by $\Phi:\nu\mapsto\widetilde{\nu}$.

\begin{example}
Let $\nu=2255431134$. Then, we have 
\begin{align}
\label{eq:exMAJINV}
\begin{aligned}
\mathrm{maj}:55\rightarrow 5544 &\rightarrow 554334 &\rightarrow 22554334 &\rightarrow 2255431134, \\
\mathrm{inv}:55\rightarrow 5544 &\rightarrow 533544 &\rightarrow53223544 &\rightarrow 5113223544.
\end{aligned}
\end{align}
The top line in Eq. (\ref{eq:exMAJINV}) is $2$-permutations of the set $[i,5]$ for $i=5,4,\ldots1$.
The bottom line in Eq. (\ref{eq:exMAJINV}) is the corresponding $2$-permutations.
As a result, we have $\Phi:2255431134\mapsto 5113223544$
with $\mathrm{MAJ}(2255431134)=\mathrm{INV}(5113223544)=6$.
\end{example}

\begin{proof}[Proof of Theorem \ref{thrm:INVMAJ}]
A map $\Phi:\mathfrak{S}_{n}^{(r)}\rightarrow\mathfrak{S}_{n}^{(r)}$ is bijective 
since we have an inverse $\Phi^{-1}$ by reversing the procedure.
Further, by definition, $\Phi$ satisfies Eq. (\ref{def:Phi}). 
As a consequence, we have
\begin{align*}
\sum_{\nu\in\mathfrak{S}^{(r)}_{n}}q^{\mathrm{MAJ}(\nu)}&=\sum_{\nu\in\mathfrak{S}^{(r)}_{n}}q^{\mathrm{INV}(\Phi(\nu))}
=\sum_{\nu\in\mathfrak{S}^{(r)}_{n}}q^{\mathrm{INV}(\nu)}, \\
&=\prod_{j=1}^{n-1}[jr+1].
\end{align*}
\end{proof}

For the $231$-avoiding $r$-Stirling permutations, we have the following proposition.
\begin{prop}
\label{prop:231Crn}
We have
\begin{align}
\label{eq:231Crn}
\sum_{\nu\in\mathfrak{S}^{(r)}_{n}(231)}q^{\mathrm{INV}(\nu)}=q^{r\genfrac{(}{)}{0pt}{}{n}{2}}C_{n}^{(r)}(q^{-1}).
\end{align}
\end{prop}
\begin{proof}
Given $\nu\in\mathfrak{S}^{(r)}_{n}(231)$, we define a sequence $\mathbf{h}(\nu):=(h_1,\ldots,h_{n})$ 
of non-negative integers by
\begin{align*}
h_i:=\genfrac{}{}{}{}{1}{r}\cdot\#\{j>i : \text{ $j$ is left to $i$ in $\nu$}\}.
\end{align*}
It is easy to see that $h_{n}=0$ and $0\le h_{i}\le r(n-i)$.
Since $\nu$ is $231$-avoiding, we have $h_{i}<h_{i+1}+r+1$.
Note that $\sum_{i}h_i=\mathrm{INV}(\nu)$.
Let $\kappa(\nu)$ be an $r$-Dyck path obtained from $\nu$ as in Section \ref{sec:bij231}.
Then, the integer $h_{i}$ is the number of unit boxes in the $i$-th row from top.
This implies that $\sum_{i}h_{i}$ is equal to the number of unit boxes above 
$(01^r)^n$ and below $\kappa(\nu)$.
As a consequence, we have Eq. (\ref{eq:231Crn}).
\end{proof}

\section{Non-crossing partitions and \texorpdfstring{$q$}{q}-Catalan number}
\label{sec:NC}
\subsection{Non-crossing partitions and major index}
A {\it non-crossing partition} of the set $[n]:=\{1,2,\ldots,n\}$ is 
a partition $\pi$ of $[n]$ such that if four integers satisfy $a<b<c<d$,
a block $B_1$ contains $a$ and $c$, and another block $B_2$ contains $b$ 
and $d$, then the two blocks $B_1$ and $B_2$ coincide with each other. 
We denote by $\mathtt{NC}(n)$ the set of non-crossing partitions of $[n]$.
It is well-known that $|\mathtt{NC}(n)|$ is given by the Catalan number.

Suppose that $\pi$ consists of $m$ blocks $B_1,\ldots,B_{m}$, and 
each block $B$ consists of ordered integers $(j_{1},\ldots,j_{p})$
where $p$ is the number of integers in $B_{i}$.

G.~Kreweras initiated a systematical study of the non-crossing partition lattice, which we call 
Kreweras lattice, in \cite{Kre72}.
The Kreweras lattice is a graded lattice with the rank function $\mathtt{rk}$.
We define the rank of $\pi$ as $\mathtt{rk}(\pi):=n-m$ where $m$ is the number of blocks in $\pi$.
When a block $B$ consists of integers $i_1,\ldots,i_m$, we write $i_1\ldots i_{m}$.
Further, since a non-crossing partition $\pi$ consists of $m$ blocks, we write 
$\pi=B_1/B_2/\ldots/B_{m}$ as a one-line notation. 
Since $\pi$ is a collection of blocks, the order of blocks in the one-line notation 
is irrelevant.
A non-crossing partition $\pi\in\mathtt{NC}(c)$ has a graphical presentation such that it has $n$ 
points on a circle, the integers in a block $B_{i}$ are connected by arches.
The non-crossing property implies that arches in different blocks do not intersect.

\begin{example}
We write $\pi=134/2/58/67$ since $\pi$ consists of four blocks 
$B_1=(1,3,4)$, $B_2=(2)$, $B_3=(5,8)$ and $B_{4}=(6,7)$.
The rank of $\pi$ is $\mathtt{rk}(\pi)=8-4=4$.
A graphical presentation of $\pi$ is given in Figure \ref{fig:NC}.
\begin{figure}[ht]
\begin{tikzpicture}[scale=0.6]
\draw circle(3cm);
\foreach \a in {0,45,90,...,315}
\filldraw [black](\a:3cm)circle(1.5pt);
\draw(90:3cm)node[anchor=south]{$1$};
\draw(45:3cm)node[anchor=south west]{$2$};
\draw(0:3cm)node[anchor=west]{$3$};
\draw(-45:3cm)node[anchor=north west]{$4$};
\draw(-90:3cm)node[anchor=north]{$5$};
\draw(135:3cm)node[anchor=south east]{$8$};
\draw(180:3cm)node[anchor=east]{$7$};
\draw(225:3cm)node[anchor=north east]{$6$};
\draw(90:3cm)to [bend right=30](0:3cm)to[bend right=30](-45:3cm)to[bend left=20](90:3cm);
\draw(-90:3cm)to[bend right=30](135:3cm);
\draw(-135:3cm)to[bend right=30](-180:3cm); 
\end{tikzpicture}
\caption{A graphical presentation of $\pi=134/2/58/67$.}
\label{fig:NC}
\end{figure}
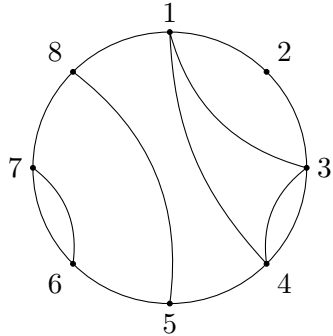
Integers in a block are connected by arches.
\end{example}

The next theorem is one of the main results of this paper.
We write $B\in\pi$ if $B$ is a block of $\pi$, and $B=(j_1,\ldots,j_{p})$ if $B$ consists 
of integers $i_1,\ldots,i_{p}$.
 
\begin{theorem}
\label{thrm:NCtoCat}
We have 
\begin{align*}
\sum_{\pi\in\mathtt{NC}(n)}
\prod_{B\in\pi}q^{\mathrm{wt}(B)}=\genfrac{}{}{}{}{1}{[n+1]}\genfrac{[}{]}{0pt}{}{2n}{n},
\end{align*}
where
\begin{align}
\label{eq:wtBi}
\mathrm{wt}(B):=
\begin{cases}
0, & \text{ if } p=1, \\
j_{1}+j_{p}+2\sum_{k=2}^{p-1}j_{k}-p+1, & \text{otherwise}
\end{cases}
\end{align}
for $B=(j_1,\ldots,j_{p})$.
\end{theorem}

The rest of this section is devoted to the proof of Theorem \ref{thrm:NCtoCat}.

We denote by $\mathtt{NC}(n;r)$ the set of non-crossing partitions $\pi$ of $n$ such 
that $\mathtt{rk}(\pi)=r$.
Let $N_2(n,r)$ be the generating function 
\begin{align*}
N_2(n,r):=\sum_{\pi\in\mathtt{NC}(n;r)}\prod_{B\in\pi}q^{\mathrm{wt}(B)}.
\end{align*}

\begin{lemma}
The generating function $N_2(n,r)$ satisfies the following recurrence relation:
\begin{align}
\label{eq:recN2}
N_2(n,r)
=q^{2r}N_2(n-1,r)+\sum_{j=0}^{n-2}\sum_{s=0}^{j-1}
N_{2}(j,s)N_2(n-j-1,r-s-1)q^{2(j+1)(r-s)-j+2s}.
\end{align}
\end{lemma}
\begin{proof}
Let $\pi\in\mathtt{NC}(n;r)$. We divide $\pi$ into two non-crossing partitions 
$\pi_{1}$ and $\pi_2$ where $\pi_1$ is a unique block in $\pi$ which contains $1$, and 
$\pi_2$ consists of the remaining blocks.
For example, if $\pi=145/23/68/7$, we have $\pi_1=145$ and $\pi_2=23/68/7$.

Suppose that $\pi_1$ consists of only the integer $1$.
We have a non-crossing partition in $\mathtt{NC}(n-1;r)$ if we decrease the integers 
of $\pi_2$ by one.
From Eq. (\ref{eq:wtBi}), this gives $q^{2r}N_2(n-1,r)$. 

Suppose that $|\pi_{1}|\ge2$ and $j\ge2$ is the second smallest integer in $\pi_1$.
By definition, $1$ and $j$ belong to the same block $\pi_1$.
This implies that $\pi$ consists of two non-crossing partitions: 
one is in $[2,j-1]$ and the other is in $\{1\}\cup[j,n]$. In the latter non-crossing 
partition, $1$ and $j$ belong to the same block.
The non-crossing partition in $[2,j-1]$ gives $q^{2s}N_2(j-2,s)$, and 
the non-crossing partition in $[j,n]$ gives $q^{2(j-1)(r-s-1)}N_2(n-j+1,r-s-1)$.
Further, $1$ and $j$ gives the factor $q^{j}$ by Eq. (\ref{eq:wtBi}).
By taking the product of all these three contributions, we have 
$N(j-2,s)N(n-j+1,r-s-1)q^{2(j-1)(r-s)-j+2+2s}$.

By taking the sum over all non-crossing partitions of rank $r$, we obtain 
Eq. (\ref{eq:recN2}).
\end{proof}

We introduce the $q$-analogue of Narayana numbers $N(n,r)$ by 
\begin{align}
\label{def:qNara}
N(n,r):=\genfrac{}{}{}{}{1}{[n]}\genfrac{[}{]}{0pt}{}{n}{r}\genfrac{[}{]}{0pt}{}{n}{r+1}q^{r(r+1)}.
\end{align}

Recall that a Dyck path $d$ is expressed as a word of $0$ and $1$. 
We say that a partial path $10$ in $d$ is a valley.
We denote by $\mathcal{P}_{n,r}$ the set Dyck paths of length $n$ with $r$ valleys.

Then, it is well-known that $q$-Narayana numbers are generating functions of Dyck paths of 
length $n$ with $r$ valleys.
This implies that we have 
\begin{align*}
\sum_{r=0}^{n-1}N(n,r)=\genfrac{}{}{}{}{1}{[n+1]}\genfrac{[}{]}{0pt}{}{2n}{n}.
\end{align*}

\begin{lemma}[{\cite[Eq. (4.1)]{FurHof85}}]
Let $d\in\{0,1\}^{n}$ be a Dyck path in $\mathcal{P}_{n,r}$.
Then, we have 
\begin{align*}
\sum_{d\in\mathcal{P}_{n,r}}q^{\mathrm{maj}(d)}=N(n,r).
\end{align*}
\end{lemma}

\begin{lemma}
The $q$-Narayana numbers satisfy the following recurrence relation:
\begin{align}
\label{eq:recN}
N(n,r)=q^{r}N(n-1,r)
+\sum_{j=0}^{n-2}\sum_{s=0}^{j-1}N(j,s)N(n-1-j,r-s-1)q^{2(j+1)(r-s)+s}.
\end{align}
\end{lemma}
\begin{proof}
A Dyck path $D$ of length $n$ is uniquely written as a concatenation of 
a prime Dyck path of length of $j$ and a Dyck path of length $n-j$.
We calculate the major index of $D$.

In the case of $j=n$, we have a Dyck path $d'$ of size $n-1$ from a prime Dyck 
path $d$ of size $n$ by deleting the first and last steps.
This implies that $\mathrm{maj}(d)=\mathrm{maj}(d')+r$ where $r$ is the number of 
valleys in $d$.
This gives the term $q^{r}N(n-1,r)$ in Eq. (\ref{eq:recN}).

Suppose that $1\le j\le n-1$, and the prime Dyck path of size $j$ has $s\le r-1$ valleys.
By the same argument as the case $j=n$, we have the term 
$q^{s}N(j-1,s)$ from the prime Dyck path.
The contribution of the Dyck path of length $n-j$ is given by
$N(n-j,r-s-1)q^{2j(r-s-1)}$. The factor $q^{2j(r-s-1)}$ comes from the fact that 
we have $2j$ steps left to the Dyck path of length $n-j$.
Finally, we have a factor $q^{2j}$ since we concatenate two Dyck paths and we have 
a valley between them.
As a result, these give the term $N(j-1,s)N(n-j,r-s-1)q^{2j(r-s)+s}$.

By taking the sum over all Dyck paths of size $n$ with $r$ valleys,
we obtain the expression Eq. (\ref{eq:recN}).
\end{proof}

\begin{prop}
\label{prop:N2N}
We have $N_2(n,r)=N(n,r)$.
\end{prop}
\begin{proof}
The $q$-Narayana numbers $N(n,r)$ are generating functions of Dyck paths of length $n$ 
with $r$ descents. Similarly, $N_2(n,r)$ is a generating function of non-crossing partition
of $[n]$ with rank $r$.
We prove the proposition by constructing a bijection between a Dyck path with $r$ descents
and a non-crossing partition with rank $r$.

We compare the expression (\ref{eq:recN2}) with Eq. (\ref{eq:recN}).
For $n=1$, we have a unique Dyck path $\mu=UD$ and a unique non-crossing 
partition $\pi=1$. We construct a bijection between a Dyck path and a non-crossing 
partition inductively. 

The first term in Eqs. (\ref{eq:recN2}) and (\ref{eq:recN}) consists of a factor 
times the generating function with smaller size $n-1$.
By induction assumption, we have a bijection between $\mathcal{P}^{1}_{n-1}$
and $\mathtt{NC}(n-1)$.
The factor in the first term differs by $q^{r}$. 
Note that a Dyck path $\mu$ in $\mathcal{P}^{1}_{n-1}$ corresponds to a prime 
Dyck path in $\mathcal{P}^{1}_{n}$ by adding $0$ and $1$ to $\mu$ from left
and right. Similarly, given a non-crossing partition $\pi\in\mathtt{NC}(n-1)$, 
we increase the integers in $\pi$ by one and add a single block $1$ to the new $\pi$.
This gives a non-crossing partition in $\mathtt{NC}(n)$.
For example, if $\pi=125/34\in\mathtt{NC}(5)$, then we have a new 
non-crossing partition $1/236/45\in\mathtt{NC}(6)$.
The weight of $\pi$ is $q^{r}$ times the weight of $\mu$. 
Suppose that $10^{k}1$ is a sub-path in $\mu$ such that the first $1$ 
is a descent in $\mu$. 
We transform the sub-path $10^{k}1$ to a new sub-path $110^{k}$ in $\mu$.
We perform this transformation for all $r$ descents in $\mu$, and this 
gives the factor $q^{r}$.
This operation is well-defined since the path $\mu$ is prime.

We consider the second terms in Eqs. (\ref{eq:recN2}) and (\ref{eq:recN}).
By induction assumption, we have a bijection between a pair of non-crossing 
partitions $(\pi_1,\pi_{2})$ and a pair of Dyck paths $(\mu_1,\mu_2)$ such 
that $\pi_1\in\mathtt{NC}(j)$ with rank $s$, $\pi_2\in\mathtt{NC}(n-j-1)$ with 
rank $r-s-1$, $\mu_1\in\mathcal{P}^{1}_{j}$ with $s$ descents, and 
$\mu_2\in\mathcal{P}^{1}_{n-j-1}$ with $r-s-1$ descents.
From the pair $(\pi_1,\pi_2)$, we construct a non-crossing partition $\pi$ in $\mathtt{NC}(n)$
as follows.
First, we increase the integers in $\pi_1$ by one, and add a block consisting of 
a single integer $1$ to the new $\pi_1$. We have a non-crossing partition $\pi'_1$ of size $j+1$.
We increase the integers in $\pi_2$ by $j+1$, and add blocks $\pi'_{1}$ to $\pi_2$.
As a result, we have a non-crossing partition $\pi$ from $\pi'_1$ and $\pi_2$.
For example, when $\pi_1=14/23$ and $\pi_2=15/24/3$, then we have $\pi'_{1}=1/25/34$ and 
$\pi=1/25/34/6\underline{10}/79/8\in\mathtt{NC}(10)$.
Similarly, we construct a Dyck path $\mu$ in $\mathcal{P}^1_{n}$ from the pair $(\mu_1,\mu_2)$
as follows.
First, we add $0$ and $1$ to $\mu_1$ from left and right respectively, and obtain a new path 
$\mu'_{1}$. The path $\mu$ is given by a concatenation of two paths $\mu'_1$ and $\mu_2$.
By comparing the factor in Eqs. (\ref{eq:recN2}) and (\ref{eq:recN}), the weight of $\pi$ 
is $q^{-j+s}$ times the wight of $\mu$. 
Therefore, to construct the bijection which preserves the weight, we transform $\mu$ 
to a new Dyck path $\overline{\mu}$. 
Since we have $j-s\ge0$, we have $\mathrm{maj}(\overline{\mu})=\mathrm{maj}(\mu)-j+s$.
Recall that $\mu$ is a concatenation of two paths $\mu'_1\in\mathcal{P}^{1}_{j}$ and $\mu_2$, and 
the last letter of $\mu'_{1}$ is $1$ which is a descent in $\mu$.
Let $\nu=110^{k}1$ be a sub-path in $\mu$ such that the second $1$ in $\nu$ is the last letter of $\mu'_1$.
We transform this sub-path $110^{k}1$ to $10^{k}11$ in $\mu$. This operation is well-defined 
since $\mu'_1$ is prime. The major index of the new path is the major index of $\mu$ minus 
one by construction. Further, the number of descents in the new path is equal to that in $\mu$.
We perform the same operations $t\le j-s$ times on the new descent $d$ until we have a pattern $010^{k}1$.
If $t=j-s$, then the we define the new path as $\overline{\mu}$.
If $t<j-s$, we move to the right-most descent in the new path which is left to the descent $d$.
Then, we continue to change the pattern $110^{k}1$ to $10^{k}11$ for some $k$, and move a left descent 
if necessary. The total number of transformations is $j-s$. 
We define the new path obtained by $j-s$ transformations as $\overline{\mu}$.
For example, we have the following Dyck path shown in Figure \ref{fig:transDyck}
if $(j,s)=(4,1)$.
\begin{figure}[ht]
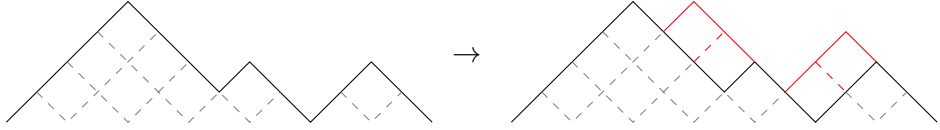

\tikzpic{-0.5}{[scale=0.4]
\draw(0,0)--(4,4)--(7,1)--(8,2)--(10,0)--(12,2)--(14,0);
\draw[gray,dashed](1,1)--(2,0)--(5,3)(2,2)--(4,0)--(6,2)
(3,3)--(6,0)--(7,1)--(8,0)--(9,1)(11,1)--(12,0)--(13,1);
}
$\rightarrow$
\tikzpic{-0.5}{[scale=0.4]
\draw(0,0)--(4,4)--(7,1)--(8,2)--(10,0)--(12,2)--(14,0);
\draw[gray,dashed](1,1)--(2,0)--(5,3)(2,2)--(4,0)--(6,2)
(3,3)--(6,0)--(7,1)--(8,0)--(9,1)(11,1)--(12,0)--(13,1);
\draw[red](5,3)--(6,4)--(8,2)(9,1)--(11,3)--(12,2);
\draw[red,dashed](6,2)--(7,3)(10,2)--(11,1);
}
\caption{Transformation of a Dyck path}
\label{fig:transDyck}
\end{figure}
The Dyck path $0000111.011.0011$ is transformed to the path 
$00001.0111.00111$ as in Figure \ref{fig:transDyck}.
Note that the number of descents is preserved.
Once we have a pair of integers $(j,s)$, the above mentioned operation
is invertible. 
As a consequence a map between $\pi$ and $\overline{\mu}$ is a bijection
which preserves the major index.
Then, we have 
\begin{align*}
N_{2}(n,r)=\sum_{\pi\in\mathtt{NC}(n)}q^{\mathrm{maj}(\pi)}
=\sum_{\overline{\mu}\in\mathcal{P}^{1}_{n}}q^{\mathrm{maj}(\overline{\mu})}
=N(n,r),
\end{align*}
which completes the proof.
\end{proof}

\begin{example}
We consider the correspondence between $N(4,2)$ and $N_2(4,2)$ given 
in Proposition \ref{prop:N2N}.
First, in $N(3,2)$, we have 
\begin{align*}
\tikzpic{-0.5}{[scale=0.3]
\draw(0,0)--(2,2)--(3,1)--(4,2)--(5,1)--(6,2)--(8,0);
\draw[dashed](1,1)--(2,0)--(3,1)--(4,0)--(5,1)--(6,0)--(7,1);
}\rightarrow
\tikzpic{-0.5}{[scale=0.3]
\draw(0,0)--(2,2)--(4,0)--(5,1)--(6,0)--(7,1)--(8,0);
\draw[dashed](1,1)--(2,0)--(3,1);
}\leftrightarrow \quad 1/234
\end{align*}
From $N(0,0)N(3,1)$, we have 
\begin{align*}
\tikzpic{-0.5}{[scale=0.3]
\draw(0,0)--(1,1)--(2,0)--(3,1)--(4,0)--(6,2)--(8,0);
\draw[dashed](5,1)--(6,0)--(7,1);
}\leftrightarrow 123/4 \qquad
\tikzpic{-0.5}{[scale=0.3]
\draw(0,0)--(1,1)--(2,0)--(4,2)--(5,1)--(6,2)--(8,0);
\draw[dashed](3,1)--(4,0)--(5,1)--(6,0)--(7,1);
}\leftrightarrow 124/3 \qquad
\tikzpic{-0.5}{[scale=0.3]
\draw(0,0)--(1,1)--(2,0)--(4,2)--(6,0)--(7,1)--(8,0);
\draw[dashed](3,1)--(4,0)--(5,1);
}\leftrightarrow 12/34 
\end{align*}
From $N(1,0)N(2,1)$, we have
\begin{align*}
\tikzpic{-0.5}{[scale=0.3]
\draw(0,0)--(2,2)--(4,0)--(5,1)--(6,0)--(7,1)--(8,0);
\draw[dashed](1,1)--(2,0)--(3,1);
}\rightarrow
\tikzpic{-0.5}{[scale=0.3]
\draw(0,0)--(2,2)--(3,1)--(4,2)--(6,0)--(7,1)--(8,0);
\draw[dashed](1,1)--(2,0)--(3,1)--(4,0)--(5,1);
}\leftrightarrow \qquad 134/2
\end{align*}
Finally, from $N(2,1)N(1,0)$, we have 
\begin{align*}
\tikzpic{-0.5}{[scale=0.3]
\draw(0,0)--(2,2)--(3,1)--(4,2)--(6,0)--(7,1)--(8,0);
\draw[dashed](1,1)--(2,0)--(3,1)--(4,0)--(5,1);
}\rightarrow
\tikzpic{-0.5}{[scale=0.3]
\draw(0,0)--(2,2)--(3,1)--(4,2)--(5,1)--(6,2)--(8,0);
\draw[dashed](1,1)--(2,0)--(3,1)--(4,0)--(5,1)--(6,0)--(7,1);
}\leftrightarrow \qquad 14/23
\end{align*}
\end{example}

\begin{proof}[Proof of Theorem \ref{thrm:NCtoCat}]
We have 
\begin{align*}
\sum_{\pi\in\mathtt{NC}(n)}\prod_{B\in\pi}q^{\mathrm{wt}(B)}
&=\sum_{r=0}^{n-1}\sum_{\pi\in\mathtt{NC}(n,r)}\prod_{B\in\pi}q^{\mathrm{wt}(B)}, \\
&=\sum_{r=0}^{n-1}N_{2}(n,r)
=\sum_{r=0}^{n-1}N(n,r),\\
&=\genfrac{}{}{}{}{1}{[n+1]}\genfrac{[}{]}{0pt}{}{2n}{n},
\end{align*}
where we have used Proposition \ref{prop:N2N} and the property of $q$-Narayana numbers.
\end{proof}

\begin{example}
When $n=3$, the major indices for non-crossing partitions are given in Table \ref{table:NCmaj}.
\begin{table}[ht]
\begin{tabular}{c|ccccc}
$\pi$ & $1/2/3$ & $12/3$ & $13/2$ & $1/23$ & $123$ \\ \hline \\[-12pt]
$\mathrm{maj}(\pi)$ & $0$ & $2$ & $3$ & $4$ & $6$     
\end{tabular}
\\[12pt]
\caption{Non-crossing partitions and major indices}
\label{table:NCmaj}
\end{table}
\end{example}

\subsection{Non-crossing partitions and \texorpdfstring{$q$}{q}-Catalan number \texorpdfstring{$C_{n}(q)$}{Cn(q)}}
Let $\pi$ be a non-crossing partition in $\mathtt{NC}(n)$.
The partition $\pi$ consists of $m$ blocks $B_1,\ldots,B_{m}$.
We define the weight $\mathrm{wt}'(B)$ of a block $B$ by 
\begin{align*}
\mathrm{wt}'(B):=\sum_{j\in B}(j-\min B),
\end{align*}
where $j\in B$ means that an integer $j$ appears in the block $B$.
Similarly, the weight $\mathrm{wt}'(\pi)$  of the partition $\pi$ is defined to be
\begin{align*}
\mathrm{wt}'(\pi):=\sum_{i=1}^{m}\mathrm{wt}'(B_{i}).
\end{align*}

\begin{prop}
\label{prop:NCqCat}
We have 
\begin{align}
\label{eq:NCqCat}
\sum_{\pi\in\mathtt{NC}(n)}q^{\mathrm{wt}'(\pi)}=q^{\genfrac{(}{)}{0pt}{}{n}{2}}C_{n}(q^{-1}).
\end{align}
\end{prop}
\begin{proof}
Let $\widetilde{C}_{n}(q)$ be the left hand side of Eq. (\ref{eq:NCqCat}).
Let $B$ be a block of $\pi$ which contains $n$ and $j:=\min B$.
We consider two cases: 1) $j=n$, and 2) $j<n$.

Case 1). Since $n$ forms a block by itself, the contribution to $\widetilde{C}_{n}(q)$ is $\widetilde{C}_{n-1}(q)$.

Case 2). Since $j$ is the minimum integer in $B$, and $\pi$ is non-crossing, blocks in $[1,j-1]$ do not intersect 
with $B$. This implies that the contribution to $\widetilde{C}_{n}(q)$ is given by $C_{j-1}(q)C_{n-j}(q)q^{n-j}$
where $q^{n-j}$ comes from the weight for integers $n$ and $j$.

We have 
\begin{align*}
\widetilde{C}_{n}(q)&=\widetilde{C}_{n-1}(q)+\sum_{j=1}^{n-1}q^{n-j}\widetilde{C}_{j-1}(q)\widetilde{C}_{n-j}(q), \\
&=\sum_{k=0}^{n}q^{k}\widetilde{C}_{k}(q)\widetilde{C}_{n-k}(q).
\end{align*}
By comparing the above expression with Eq. (\ref{eq:recCn}), we obtain Eq. (\ref{eq:NCqCat}).
\end{proof}

\subsection{Refinement of the major index}
Let $\pi$ be a non-crossing partition which consists of $m$ blocks $B_1,\ldots,B_{m}$.
The major index of $\pi$ is defined as the sum of the major index of $B_{i}$, {\it i.e.},
$\mathrm{maj}(\pi)=\sum_{i=1}^{m}\mathrm{wt}(B_{i})$ where $\mathrm{wt}(B_{i})$ is defined 
by Eq. (\ref{eq:wtBi}).
We will construct a bijection $\beta:\mathtt{NC}(n)\xrightarrow{\sim}\mathcal{S}(312)$, 
$\pi\mapsto\sigma:=\beta(\pi)$, from 
non-crossing partitions to $312$-avoiding permutations such that 
\begin{align*}
\mathrm{maj}(\pi)=\mathrm{maj}(\sigma)+\mathrm{maj}(\sigma^{-1}).
\end{align*}

Let $S_{X}(\pi):=(i_1,\ldots,i_{m})$ and $S_{Y}(\pi):=(j_1,\ldots,j_{m})$ be two increasing sequences 
of positive integers obtained from $\pi$ by 
\begin{align*}
\{i_{k}:1\le k\le m\}&=\{\min B_{k}:1\le k\le m\}, \\
\{j_{k}:1\le k\le m\}&=\{\max B_{k}:1\le k\le m\}.
\end{align*}
By definition, we have $m=n-\mathtt{rk}(\pi)$ where $\mathtt{rk}$ is the rank function.
In a Dyck path, a subsequence $01$ is called a peak.
We say that the position of a peak in a Dyck path is the $i$-th position 
if the $i$-th step in the Dyck path is $0$ and the $i+1$-th step is $1$.
It is easy to see that if we fix the positions of peaks and valleys, these give 
a unique Dyck path.
We give the positions of peaks in a Dyck path $d(\pi)$ from $S_{X}(\pi)$ and
$S_{Y}(\pi)$ by 
\begin{align*}
S_{P}(\pi):=\{i_{k}+j_{k}-1: 1\le k\le m\}.
\end{align*}
Similarly, the positions of valleys in a Dyck path $d(\pi)$ is given by
\begin{align*}
S_{V}(\pi):=\{i_{k+1}+j_{k}-1: 1\le k\le m-1\}.
\end{align*}
Note that the number of valleys is one less than that of peaks in a Dyck path.
Let $d(\pi)$ be the Dyck path whose positions of peaks and valleys are given 
by $S_{P}(\pi)$ and $S_{V}(\pi)$.

Let $d:=d(\pi)=d_1\ldots d_{2n}\in\{0,1\}^{2n}$ be a Dyck path in terms of two 
alphabets $0$ and $1$.
Let $l_{i}$, $1\le i\le n$, be the position of the $i$-th $0$ from the left end in $d$.
We define non-negative integers $a_{i}$, $1\le i\le n$, by 
\begin{align*}
a_{i}:=\{j<l_{n+1-i}: d_{j}=1\}.
\end{align*}
Then, we define a permutation $\mu$ whose Lehmer code is $(a_1,\ldots,a_{n})$.
As a summary, we have 
\begin{align}
\label{eq:pitosigma}
\pi\mapsto d(\pi)\mapsto \mu\mapsto\overline{\mu^{-1}}=:\sigma,
\end{align}
where $\overline{\mu}=\overline{\mu_1}\ldots\overline{\mu_{n}}$ with 
$\overline{\mu_{i}}=n+1-\mu_{i}$.

\begin{lemma}
The permutation $\sigma$ is $312$-avoiding.
\end{lemma}
\begin{proof}
Since $d(\pi)$ is a Dyck path, the Lehmer code $(a_1,\ldots,a_{n})$ satisfies 
$a_1\ge a_2\ge \ldots \ge a_{n}$. The permutation $\mu$ is a $132$-avoiding 
permutation. This implies that $\mu^{-1}$ is also $132$-avoiding, and $\sigma$
is $312$-avoiding.
\end{proof}

The major indices of $\sigma$ and $\sigma^{-1}$ are given in terms of 
the two sets $S_{X}(\pi)$ and $S_{Y}(\pi)$.
\begin{lemma}
\label{lemma:sigmainv}
We have
\begin{align*}
\mathrm{maj}(\sigma)=\sum_{i\in A(\pi)}i, \qquad
\mathrm{maj}(\sigma^{-1})=\sum_{i\in B(\pi)}i,
\end{align*}
where 
\begin{align*}
A(\pi):=\{i<n: i+1\notin S_{X}(\pi) \}, \qquad
B(\pi):=\{i<n: i\notin S_{Y}(\pi)\}.
\end{align*}
\end{lemma}
\begin{proof}
We first show $\mathrm{maj}(\sigma)=\sum_{i\in A(\pi)}i$.
We consider the Dyck path $\mu$ in Eq. (\ref{eq:pitosigma}). 
The integer $i_{k}$ in $S_{X}(\pi)$ indicates the column number of 
the $k$-th peak in $\mu$. 
By construction of $\sigma$, this means that $i_k$ cannot be 
a descent in $\sigma$. On the other hand, if the integer $i+1\notin S_{X}(\pi)$
and $i<n$, the integer $i$ is a descent of $\sigma$.
Since $\mathrm{maj}(\sigma)$ is the sum of the descents of $\sigma$, the 
statement follows.
  
The inverse of the permutation corresponds to the exchange of $X$ and $Y$, and 
the exchange of row and column in $\mu$.
By a similar argument in the case of $\mathrm{maj}(\sigma)$, we have 
$\mathrm{maj}(\sigma^{-1})=\sum_{i\in B(\pi)}i$.
This completes the proof.
\end{proof}

\begin{theorem}
\label{thrm:NCqt}
Let $\pi\in\mathtt{NC}(n)$, and $\sigma$ is the $312$-avoiding permutation 
obtained by Eq. (\ref{eq:pitosigma}).
We have 
\begin{align*}
\mathrm{maj}(\pi)=\mathrm{maj}(\sigma)+\mathrm{maj}(\sigma^{-1}).
\end{align*}
\end{theorem}
\begin{proof}
From Lemma \ref{lemma:sigmainv}, we have 
\begin{align*}
\mathrm{maj}(\sigma)&=\genfrac{}{}{}{}{1}{2}n(n+1)-\sum_{i=1}^{m}\min B_{i}-\mathtt{rk}(\pi),\\
\mathrm{maj}(\sigma^{-1})&=\genfrac{}{}{}{}{1}{2}n(n+1)-\sum_{i=1}^{m}\max B_{i},
\end{align*}
where $\mathtt{rk}(\pi)$ is the rank of $\pi$.
We compare these with Eq. (\ref{eq:wtBi}), namely, we have 
\begin{align*}
\mathrm{wt}(\pi)&=\sum_{i=1}^{m}\mathrm{wt}(B_{i}), \\
&=\sum_{i=1}^{m}\left(2\sum_{k=1}^{d(i)}j_{k}^{i}-\min B_{i}-\max B_{i} -d(i)+1\right), \\
&=n(n+1)-\sum_{i=1}^{m}(\min B_i+\max B_{i})-\mathtt{rk}(\pi), \\
&=\mathrm{maj}(\sigma)+\mathrm{maj}(\sigma^{-1}),
\end{align*}
where $B_{i}=(j_{1}^{i},\ldots,j_{d(i)}^{i})$ is a block and we have used $\mathtt{rk}(\pi)=\sum_{i=1}^{m}(d(i)-1)$.
This completes the proof.
\end{proof}

\begin{example}
We consider a non-crossing partition $\pi=158/24/3/67$.
The major index of $\pi$ is $\mathrm{maj}(\pi)=17+5+12=34$.
We have $S_{X}(\pi)=(1,2,3,6)$ and $S_{Y}(\pi)=(3,4,7,8)$.
Then, the set of the positions of peaks and valleys are 
given by $S_{P}(\pi)=\{3,5,9,13\}$ and $S_{V}(\pi)=\{4,6,12\}$.
We obtain the Dyck path $d(\pi)=0001.01.000111.0111$.
From this, we have $\mu=63452178$, which gives 
$\sigma=34765821$ and $\sigma^{-1}=87125436$.
The major indices of $\sigma$ and $\sigma^{-1}$ are given by
$\mathrm{maj}(\sigma)=3+4+6+7=20$ and 
$\mathrm{maj}(\sigma^{-1})=1+2+5+6=14$.
Therefore, we have $\mathrm{maj}(\pi)=34=\mathrm{maj}(\sigma)+\mathrm{maj}(\sigma^{-1})$.
\end{example}

We define the generating function $\mathcal{A}(q,t)$ by 
\begin{align*}
\mathcal{A}(q,t):=\sum_{\sigma\in S(312)}q^{\mathrm{maj}(\sigma)}t^{\mathrm{maj}(\sigma^{-1})}.
\end{align*}

The following is a direct consequence of Theorem \ref{thrm:NCqt}.
\begin{cor}
\label{cor:Aqq}
We have 
\begin{align*}
\mathcal{A}(q,q)=\genfrac{}{}{}{}{1}{[n+1]}\genfrac{[}{]}{0pt}{}{2n}{n}.
\end{align*}
\end{cor}

\begin{remark}
Two remarks are in order:
\begin{enumerate}
\item
If we define $\mathcal{A}'(q,t):=t^{\genfrac{(}{)}{0pt}{}{n}{2}}\mathcal{A}(q,t^{-1})$, 
the generating function $\mathcal{A}'(q,t)$ is the polynomial $\mathcal{A}(q,t)$ studied 
in \cite{Stu09}.
Although the polynomial $\mathcal{A}'(q,t)$ is not the same as the $q,t$-Catalan numbers introduced 
by A.~Garsia nad M.~Haiman in \cite{GarHai96}, they coincide with each other at $t=q^{-1}$.

\item
Corollary \ref{cor:Aqq} implies that we construct a bijection between $231$-avoiding 
permutations and non-crossing partitions which sends the sum of the major index and 
the inverse major index of a $231$-avoiding permutation to the major index of the
corresponding non-crossing partition.
\end{enumerate}
\end{remark}

\section{Dyck tilings and \texorpdfstring{$q$}{q}-Catalan number}
\label{sec:DT}
\subsection{Dyck tilings and \texorpdfstring{$231$}{231}-avoiding permutations}
In this section, we consider a relation between the $q$-Catalan numbers and 
combinatorial objects called Dyck tilings.

We introduce the notion of Dyck tilings following \cite{KW11,KMPW12,SZJ12}. 
A standard notion of a Dyck tiling is a tiling in an area between two Dyck paths.
We consider another type of Dyck tilings studied in \cite[Section 3]{S19}.

Let $\lambda_{b}=(01)^{n}$ and $\lambda_{t}=1^{2n}$ be two paths where $0$ 
(resp. $1$) stands for north-east (resp. south-east) step.
We put these two paths in the Cartesian coordinate such that $\lambda_{b}$
is starting from $(0,0)$ and ending at $(2n,0)$, and $\lambda_{t}$ is starting 
$(0,2n)$ and ending at $(2n,0)$.
We consider the region $R$ surrounded by three lines $\lambda_{b}$, $\lambda_{t}$ 
and $x=0$. We put squares which are rotated $45$ degrees in the region $R$.
We put positive integers from $1$ to $n$ on the squares from top to bottom in the line $x=0$.
We call a square with a label an anchor box.

To define a Dyck tiling in the region $R$, we introduce Dyck tiles first.
A skew shape is said to be a ribbon if it does not contain a two-by-two 
squares.
A Dyck tile is a ribbon such that if we connect the centers of the squares by a line, then 
the line is a Dyck path.
A size of a Dyck tile is defined to be the size of the Dyck path which characterizes the 
Dyck tile. 
For example, a unit square is a Dyck tile of size zero.
We call the Dyck tile of size zero a trivial Dyck tile, and other Dyck tiles are called 
non-trivial Dyck tiles.
We consider a tiling of the region $R$ by Dyck tiles which satisfies the following 
conditions:
\begin{enumerate}[($\diamondsuit$1)]
\item  The tiling is cover-inclusive. This means that if we move a Dyck tile downward 
by $(0,-2)$, then it is contained in another Dyck tile or below the path $\lambda_{b}$.
\end{enumerate}
The condition ($\diamondsuit$1) implies that we have only Dyck tiles which are characterized 
by zig-zag Dyck paths $(01)^{l}$ with some $l\ge0$ since the bottom path $\lambda_{b}$ is a 
zig-zag path. 

Let $D$ be a cover-inclusive Dyck tiling in the region $R$. 
We consider an Hermite history of $D$ following \cite{KMPW12,S19,S20}.
The boundary of a Dyck tile consists of up steps $(1,1)$ and down steps $(-1,1)$.
Given a Dyck tile $d$, we connect the left-most up step and the right-most up step
by a line such that it goes through inside of $d$.
We call the line inside of $d$ a trajectory.
We call the left-most up step an entry and the right-most up step an exit. 
Since $D$ is a cover-inclusive Dyck tiling, 
A exit of a Dyck tile shares the same edge as an entry of another Dyck tile, 
an up step in $\lambda_{b}$, or an up step (which may not be an entry) of another Dyck tile.
We connect trajectories of Dyck tiles into a longer trajectory if possible.
Recall that an anchor box is on the line $x=0$, and there is no Dyck tile north-west to it.
This means that one can start the trajectory from the entry of an anchor box and extend 
it by the above mentioned way.
Suppose that the trajectory starting from an anchor labeled $i$ box ends at an up step of $\lambda_{b}$.
Then, we write $i$ below the corresponding peak of $\lambda_{b}$.
If the trajectory starting from an anchor box labeled $i$ stops at an up step $u$ of a Dyck tile, 
we put $i$ at the peak of $\lambda_{b}$ which is downward $u$.
In this way, we put integers $[1,n]$ at the peaks of $\lambda_{b}$. 
If we read the integers at the peaks of $\lambda_{b}$ from right to left, we 
obtain a permutation $w$.
This gives a correspondence between a Dyck tiling in $R$ and a permutation.
Conversely, Let $w=w_1\ldots w_{n}$ be a permutation of length $n$.
Then, since $\lambda_{b}$ has $n$ peaks, we put the integer $w_{i}$ below the
$i$-th peak from right to left. 
Then, we obtain a Dyck tiling by constructing a trajectory from an up step of $\lambda_{b}$ labeled $i$
to an anchor box labeled $i$.
Therefore, we have a bijection between a Dyck tiling in $R$ and a permutation.

\begin{figure}[ht]
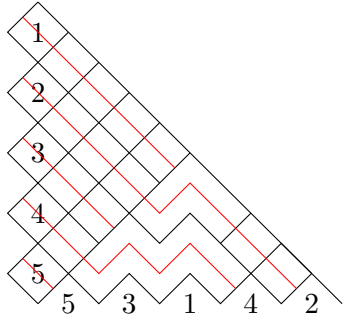

\tikzpic{-0.5}{[scale=0.4]
\draw(0,0)--(1,1)--(2,0)--(3,1)--(4,0)--(5,1)--(6,0)--(7,1)--(8,0)--(9,1)--(10,0);
\draw(0,0)--(-1,1)--(0,2)--(-1,3)--(0,4)--(-1,5)--(0,6)--(-1,7)--(0,8)--(-1,9)--(0,10);
\draw(0,2)--(1,1)(0,4)--(2,2)(0,6)--(3,3)(0,8)--(4,4)(0,10)--(9,1);
\draw(1,1)--(3,3)--(4,2)--(5,3)--(7,1)--(8,2)--(9,1)(0,2)--(4,6)(0,4)--(3,7)(0,6)--(2,8)(0,8)--(1,9);
\draw(3,3)--(5,5)(6,2)--(7,3);
\draw(0,1)node{$5$}(0,3)node{$4$}(0,5)node{$3$}(0,7)node{$2$}(0,9)node{$1$};
\draw[red](-0.5,1.5)--(0.5,0.5)(-0.5,3.5)--(2,1)--(3,2)--(4,1)--(5,2)--(6.5,0.5);
\draw[red](-0.5,5.5)--(2.5,2.5)(-0.5,7.5)--(4,3)--(5,4)--(8.5,0.5)(-0.5,9.5)--(4.5,4.5);
\draw(1,0)node{$5$}(3,0)node{$3$}(5,0)node{$1$}(7,0)node{$4$}(9,0)node{$2$};
}
\caption{An example of a Dyck tiling in the region $R$ and a permutation}
\label{fig:DyckR}
\end{figure}
An example of a Dyck tiling for the permutation $24135$ is given in 
Figure \ref{fig:DyckR}.
The red lines in Dyck tiles are trajectories.

To relate a Dyck tiling in $R$ to the major index of a Dyck path,
we introduce a bijection $\rho$ between $231$-avoiding permutation and 
a Dyck path. This bijection $\rho$ is different from the bijection $\kappa$
introduced in Section \ref{sec:bij231}.
We study the bijection $\rho$ following \cite{Stu09}.

Let $\pi:=\pi_1\ldots\pi_{n}\in\mathfrak{S}_{n}(231)$ be a $231$-avoiding permutation.
We define the descent set $\mathrm{Des}(\pi)$ of $\pi$ by
\begin{align}
\label{eq:Despi2}
\mathrm{Des}(\pi)=\{i: \pi_{i}>\pi_{i+1}\}\cup\{n\}.
\end{align}
Similarly, we define the descent set $\mathrm{iDes}(\pi)$ of $\pi^{-1}$ by 
$\mathrm{iDes}(\pi):=\mathrm{Des}(\pi^{-1})$.

\begin{remark}
Note that the descent set $\mathrm{Des}(\pi)$ contains $n$ and this is different 
from the standard definition of the descent set. 
\end{remark}

\begin{lemma}[Corollary 3.2 in \cite{Stu09}]
\label{lemma:DesiDes}
We have $|\mathrm{Des}(\pi)|=|\mathrm{iDes}(\pi)|$ if $\pi$ is a $231$-avoiding permutation.
\end{lemma}

Lemma \ref{lemma:DesiDes} plays an important role when we construct a Dyck path from 
the $231$-avoiding permutation $\pi$.
Let $\mathrm{Des}(\pi)=(i_1,\ldots,i_{m})$ and $\mathrm{iDes}(\pi)=(j_1,\ldots,j_{m})$
with some $m\ge0$.
We construct a Dyck path from $\mathrm{Des}(\pi)$ and $\mathrm{iDes}(\pi)$ from right to 
left as follows.
We first place $j_1$ $1$'s from right to left. Then, we append $i_1$ $0$'s from left to 
$1^{j_1}$. We append $j_2-j_1$ $1$'s from left, and successively append $i_2-i_1$ $0$'s 
from left. We continue to append $j_{k+1}-j_k$ $1$'s and $i_{k+1}-i_{k}$ $0$'s from 
left. Then, we have a Dyck path of size $2n$.

\begin{example}
Let $\pi=631245$. We have $\mathrm{Des}(\pi)=\{1,2,6\}$ and 
$\mathrm{iDes}(\pi)=\mathrm{Des}(342561)=\{2,5,6\}$.
Then, we have the Dyck path $00001.0111.011$.
\end{example}

\begin{defn}
\label{defn:rho}
We denote by $\rho:\mathfrak{S}_{n}(231)\xrightarrow{\sim}\mathcal{P}^1_{n}$, 
$\pi\mapsto \rho(\pi)$ the bijection defined above.
\end{defn}

\subsection{Major index on Dyck tilings}
In what follows, we define the major index on the Dyck tiling in the region $R$.
Recall that a Dyck tile $d$ is characterized by a Dyck path $p(d)$.
Suppose a Dyck tile $d$ consists of $N(d)$ unit squares.
We define the size $l(d)$ of a Dyck tile $d$ by
\begin{align*}
l(d):=
\begin{cases}
0, & \text{ if $d$ is a trivial Dyck tile}, \\
(N(d)+1)/2, & \text{ otherwise }.
\end{cases}
\end{align*}
Let $x(d)$ (resp. $y(d)$) be the $x$-coordinate (resp. $y$-coordinate) of 
the left-most point of the Dyck path $p(d)$. 
When the left-most box of a Dyck tile $d$ is an anchor box labeled $n$, 
we have $x(d)=0$ and $y(d)=1$.
We define the major index of a Dyck tile $d$ by
\begin{align}
\label{eq:majd}
\mathrm{maj}(d):=
\begin{cases}
0, & \text{ if $d$ is a trivial Dyck tile}, \\
l(d)+x(d)+(y(d)-1)/2, & \text{otherwise}.
\end{cases}
\end{align}
 
We define the major index of a Dyck tiling $D$ by
\begin{align*}
\mathrm{maj}(D):=\sum_{d\in D}\mathrm{maj}(d),
\end{align*}
where $d\in D$ means that a Dyck tile $d$ is in a Dyck tiling $D$.

In Figure \ref{fig:DyckR}, we have two non-trivial Dyck tiles of size $3$ and $2$.
The major index of the former is $3+2+0=5$, and that of the latter is $2+4+1=7$.
As a total, this Dyck tiling has the major index $5+7=12$.

The next theorem is the main result of this section.
\begin{theorem}
\label{thrm:majD}
Let $\pi\in\mathfrak{S}_{n}(231)$, $D(\pi)$ be the Dyck tiling corresponding to $\pi$, and 
$\rho$ be the bijection defined in Definition \ref{defn:rho}.
Then, we have 
\begin{align}
\label{eq:majDpi}
\mathrm{maj}(D(\pi))=\mathrm{maj}(\rho(\pi)).
\end{align}
\end{theorem}

Before proceeding to the proof of Theorem \ref{thrm:majD}, 
we introduce two lemmas used later.

Let $w\in\mathfrak{S}_{n}(231)$. Then, we write $j\leftarrow i$ (resp. $i\rightarrow j$) 
if $j$ is left (resp. right) to $i$ in $w$.
We denote by $\mathrm{pos}(i)$ the position of $i$ from left in $w$, {\it i.e.}, 
$\mathrm{pos}(i)=w^{-1}(i)$.

\begin{lemma}
Suppose that a non-trivial Dyck tile $d$ belongs to a trajectory starting from 
the anchor box labeled $i$. 
Then, we have 
\begin{align}
\label{eq:majdDyck}
\mathrm{maj}(d)=
\begin{cases}
0, & \text{if} \quad \mathrm{pos}(i)=i+\#\{j>i : j\leftarrow i\}, \\
i+1-\mathrm{pos}(i)+2(n-i), &\text{otherwise}.
\end{cases}
\end{align}
\end{lemma}
\begin{proof}
We consider the Dyck tiling corresponding to $w\in\mathfrak{S}_{n}(231)$.
Recall that, in a Dyck tiling, a trajectory starts from an anchor box labeled $i$ and ends with some point.
Since the bottom Dyck path is a zig-zag path, each trajectory contains at most one non-trivial Dyck tiles.
We compute the major index for the Dyck tile $d_{i}$ which is in a trajectory labeled $i$ by Eq. (\ref{eq:majd}).
Define the number $h(i)$ by 
\begin{align*}
h(i):=n+1-\mathrm{pos}(i)+\#\{j>i: j\leftarrow i \text{ in } w\}.
\end{align*}
Then, the trajectory labeled $i$ ends at the $h(i)$-th up step from left in $\lambda_{b}$.
It is obvious that if $h(i)=n+1-i$, equivalently $\mathrm{pos}(i)=i+\#\{j>i: j\leftarrow i \text{ in } w\}$, 
then we have no non-trivial Dyck tiles in the trajectory labeled $i$. In this case, we have 
$\mathrm{maj}(d_{i})=0$.
Suppose that $h(i)\neq n+1-i$. 
Then, the size $l(d_i)$ of the Dyck tile is given by 
\begin{align*}
l(d_{i})&=h(i)-(n+1-i)+1,\\
&=i+1-\mathrm{pos}(i)+\#\{j>i: j\leftarrow i \text{ in } w\}.
\end{align*}
The $x$-coordinate of $d_{i}$, and the $y$-coordinate of $d_{i}$ are given by
\begin{align*}
x(d)&=2\cdot \#\{j>i : i\rightarrow j \text{ in } w\}, \\
(y(d)-1)/2&=\#\{j>i: j\leftarrow i \text{ in } w\}.
\end{align*}
From Eq. (\ref{eq:majd}), we have Eq. (\ref{eq:majdDyck}).
\end{proof}

\begin{lemma}
Let $w\in\mathfrak{S}_{n}(231)$, $\mathrm{Des}(w)=\{i_1,\ldots,i_{m}\}$
and $\mathrm{iDes}(w)=\{j_1,\ldots,j_{m}\}$.
Then, we have 
\begin{align}
\label{eq:majrhow}
\mathrm{maj}(\rho(w))=\sum_{k=1}^{m-1}\{(j_m-j_{k})+(i_m-i_{k})\}.
\end{align}
\end{lemma}
\begin{proof}
By the construction of $\rho$, the Dyck path $p:=\rho(w)$ starts with 
$i_m-i_{m-1}$ $0$'s, followed by $j_{m}-j_{m-1}$ $1$'s, followed 
by $i_{m-1}-i_{m-2}$ $0$'s, followed by $i_{m-1}-i_{m-2}$ $1$'s, and so on.
A descent $i$ satisfies $p_{i}=1$ and $p_{i+1}=0$.
So, it is easy to see that the sum of descents is given by Eq. (\ref{eq:majrhow}).
\end{proof}

\begin{proof}[Proof of Theorem \ref{thrm:majD}]
We prove the theorem by induction on $n$. For $n=1$ and $n=2$, it is easy to see
that Eq. (\ref{eq:majDpi}) holds true.
We assume that Eq. (\ref{eq:majDpi}) holds true up to $n-1$.
The permutation $w\in\mathfrak{S}_{n}(231)$ is written as a concatenation 
$w=w'1w''$ in one-line notation.
Let $w_{\downarrow}$ be a permutation of size $n-1$ obtained from $w$ by deleting 
$1$ and decease all the integers by one. 
By its definition, an integer $i$ in $w_{\downarrow}$ corresponds to the integer $i+1$ in $w$.

We first compute the major index of the Dyck tiling in the trajectory labeled $i+1$ where 
the integer $i+1$ appears in $w''$.
Suppose that the integer $i+1$ satisfies $\mathrm{pos}(i+1)=i+1+\#\{j>i+1:j\leftarrow i+1\}$ in $w$.
Then, the integer $i$ in $w_{\downarrow}$ also satisfies $\mathrm{pos}(i)=i+\#\{j>i:j\leftarrow i\}$.
This means that if $i+1$ is right to $1$ in $w$ and the major index of the Dyck tiles in the trajectory 
labeled $i+1$ is zero, then the major index of the Dyck tiles in the trajectory labeled $i$ in $w_{\downarrow}$ 
is also zero.
Secondly, we compute the major indexes for the Dyck tiles in the trajectories labeled $i+1$ in $w'$.
Since $w$ is a $231$-avoiding permutation, it is easy to see that $w'$ is a decreasing sequence.
Suppose $i+1$ appears in $w'$. Then, we have $\#\{j>i: j\leftarrow i\}=\mathrm{pos}(i+1)-1$, 
which implies that $\mathrm{pos}(i+1)\neq i+1+\#\{j>i: j\leftarrow i\}$.
Therefore, there is a non-trivial Dyck tile $d_{i+1}$ in the trajectory labeled $i+1$.
From Eq. (\ref{eq:majdDyck}), we have 
\begin{align*}
\mathrm{maj}(d_{i+1})=i+2-\mathrm{pos}(i+1)+2(n-i-1).
\end{align*}
We compute the major index of a non-trivial Dyck tile $d'_{i}$ in the trajectory 
labeled $i$ in $w_{\downarrow}$:
\begin{align*}
\mathrm{maj}(d'_{i})=i+1-\mathrm{pos}(i)+2(n-1-i).
\end{align*}
Since the position of $i+1>3$ in $w$ is the same as that of $i$ in $w_{\downarrow}$, we have 
$\mathrm{maj}(d_{i+1})=\mathrm{maj}(d'_{i})+1$. 
If $2$ is in $w'$, then $\mathrm{pos}(2)=\mathrm{pos}(1)-1$ in $w$.
This means that the major index of the Dyck tile in the trajectory labeled $2$ 
is given by $2n-\mathrm{pos}(1)$.
From these observations, we have 
\begin{align}
\label{eq:majDwdif}
\mathrm{maj}(D(w))=\mathrm{maj}(D(w_{\downarrow}))
+
\begin{cases}
2n-2, & \text{ if $2$ is in $w'$}, \\
\mathrm{pos}(1)-1, & \text{ otherwise}.
\end{cases}
\end{align}

We compute the major index for the Dyck path $\rho(w)$ by use of $\mathrm{Des}(w)$ and $\mathrm{iDes}(w)$.
Given an integer $k$, we define $\mathrm{lDes}^{k}(w):=\mathrm{Des}(w)\cap[1,k]$ and 
$\mathrm{rDes}^{k}(w):=\mathrm{Des}(w)\cap[k+1,n]$.
By definition, we have $\mathrm{Des}(w)=\mathrm{lDes}^{k}(w)\cup\mathrm{rDes}^{k}(w)$.
Let $k=\mathrm{pos}(1)=w^{-1}(1)$.
It is easy to see that 
\begin{align*}
\mathrm{lDes}^{k}(w)=
\begin{cases}
\mathrm{lDes}^{k-1}(w_{\downarrow})\cup\{k-1\}, & \text{ if $2$ is in $w'$}, \\
\mathrm{lDes}^{k-1}(w_{\downarrow}), & \text{otherwise}.
\end{cases}
\end{align*}
Similarly, we have 
\begin{align*}
\mathrm{rDes}^{k}(w)=\{i+1: i\in\mathrm{rDes}^{k-1}(w_{\downarrow})\}.
\end{align*}
For $\mathrm{iDes}(w)=\mathrm{Des}(w^{-1})$, we have 
\begin{align*}
\mathrm{iDes}(w)=
\begin{cases}
\{i+1: i\in\mathrm{iDes}(w_{\downarrow})\}\cup\{1\}, & \text{ if $2$ is in $w'$}, \\
\{i+1: i\in\mathrm{iDes}(w_{\downarrow})\}, & \text{otherwise}. 
\end{cases}
\end{align*}
Recall $\mathrm{Des}(\pi)$ is defined in Eq. (\ref{eq:Despi2}).
By combining them by use of Eq. (\ref{eq:majrhow}), we have 
\begin{align}
\label{eq:majrhowdif}
\mathrm{maj}(\rho(w))=\mathrm{maj}(\rho(w_{\downarrow}))+
\begin{cases}
2n-2, & \text{ if $2$ is in $w'$}, \\
\mathrm{pos}(1)-1, & \text{otherwise},
\end{cases}
\end{align}
where we have used $|\mathrm{lDes}^{k-1}(w_{\downarrow})|=\mathrm{pos}(1)-2$.
From Eqs. (\ref{eq:majDwdif}) and (\ref{eq:majrhowdif}), we have 
$\mathrm{maj}(D(w))=\mathrm{maj}(\rho(w))$ if and only if 
$\mathrm{maj}(D(w_{\downarrow}))=\mathrm{maj}(\rho(w_{\downarrow}))$.
By use of induction hypothesis, we have Eq.(\ref{eq:majDpi}), which 
completes the proof.
\end{proof}

\begin{figure}	
\centering
\begin{tabular}{c|cccc} \hline
$S(231)$ & $1234$ & $4123$  & $1423$  &  $1243$    \\ \hline
Dyck path & 
\tikzpic{-0.5}{[scale=0.3, show background rectangle,background rectangle/.style={fill=none},inner frame sep=1mm]
\draw(0,0)--(4,4)--(8,0);
\draw[dashed](1,1)--(2,0)--(5,3)(2,2)--(4,0)--(6,2)(3,3)--(6,0)--(7,1);
}
& \tikzpic{-0.5}{[scale=0.3, show background rectangle,background rectangle/.style={fill=none},inner frame sep=1mm]
\draw(0,0)--(3,3)--(4,2)--(5,3)--(8,0);
\draw[dashed](1,1)--(2,0)--(4,2)(2,2)--(4,0)--(6,2)(4,2)--(6,0)--(7,1);
}
& 
\tikzpic{-0.5}{[scale=0.3, show background rectangle,background rectangle/.style={fill=none},inner frame sep=1mm]
\draw(0,0)--(2,2)--(3,1)--(5,3)--(8,0);
\draw[dashed](1,1)--(2,0)--(3,1)--(4,0)--(6,2)(4,2)--(6,0)--(7,1);
}&
\tikzpic{-0.5}{[scale=0.3, show background rectangle,background rectangle/.style={fill=none},inner frame sep=1mm]
\draw(0,0)--(1,1)--(2,0)--(5,3)--(8,0);
\draw[dashed](3,1)--(4,0)--(6,2)(4,2)--(6,0)--(7,1);
}  \\ \hline
Dyck tiling 
&
\tikzpic{-0.5}{[scale=0.3, show background rectangle,background rectangle/.style={fill=none},inner frame sep=1mm]
\draw(0,0)--(1,1)--(2,0)--(3,1)--(4,0)--(5,1)--(6,0)--(7,1)--(8,0);
\draw(0,0)--(-1,1)--(0,2)--(-1,3)--(0,4)--(-1,5)--(0,6)--(-1,7)--(0,8)--(7,1);
\draw(7,0)node{$1$}(5,0)node{$2$}(3,0)node{$3$}(1,0)node{$4$};
\draw(1,1)--(0,2)(3,1)--(0,4)(5,1)--(0,6)(7,1)--(0,8);
\draw(1,1)--(4,4)(3,1)--(5,3)(5,1)--(6,2);
\draw(0,2)--(3,5)(0,4)--(2,6)(0,6)--(1,7);
} 	
&
\tikzpic{-0.5}{[scale=0.3, show background rectangle,background rectangle/.style={fill=none},inner frame sep=1mm]
\draw(0,0)--(1,1)--(2,0)--(3,1)--(4,0)--(5,1)--(6,0)--(7,1)--(8,0);
\draw(0,0)--(-1,1)--(0,2)--(-1,3)--(0,4)--(-1,5)--(0,6)--(-1,7)--(0,8)--(7,1);
\draw(7,0)node{$4$}(5,0)node{$1$}(3,0)node{$2$}(1,0)node{$3$};
\draw(5,3)--(4,2)--(3,3)--(2,2)--(1,3)--(0,2);
\draw(1,3)--(0,4)(3,3)--(0,6)(0,6)--(1,7)(0,4)--(2,6)(1,3)--(3,5)(3,3)--(4,4);
\draw(3,1.8)node{$4$};
} 
&
\tikzpic{-0.5}{[scale=0.3, show background rectangle,background rectangle/.style={fill=none},inner frame sep=1mm]
\draw(0,0)--(1,1)--(2,0)--(3,1)--(4,0)--(5,1)--(6,0)--(7,1)--(8,0);
\draw(0,0)--(-1,1)--(0,2)--(-1,3)--(0,4)--(-1,5)--(0,6)--(-1,7)--(0,8)--(7,1);
\draw(7,0)node{$1$}(5,0)node{$4$}(3,0)node{$2$}(1,0)node{$3$};
\draw(5,1)--(3,3)--(2,2)--(1,3)--(0,2);
\draw(0,6)--(3,3)--(4,4)(0,4)--(1,3)--(3,5)(0,6)--(1,7)(0,4)--(2,6)(4,2)--(5,3)(5,1)--(6,2);
\draw(2,1.3)node{$3$};
} 
&
\tikzpic{-0.5}{[scale=0.3, show background rectangle,background rectangle/.style={fill=none},inner frame sep=1mm]
\draw(0,0)--(1,1)--(2,0)--(3,1)--(4,0)--(5,1)--(6,0)--(7,1)--(8,0);
\draw(0,0)--(-1,1)--(0,2)--(-1,3)--(0,4)--(-1,5)--(0,6)--(-1,7)--(0,8)--(7,1);
\draw(7,0)node{$1$}(5,0)node{$2$}(3,0)node{$4$}(1,0)node{$3$};
\draw(3,1)--(1,3)--(0,2);
\draw(0,6)--(5,1)--(6,2)(0,4)--(1,3)--(3,5)(0,6)--(1,7)(0,4)--(2,6)(2,2)--(4,4)(3,1)--(5,3);
\draw(1,1.8)node{$2$};
} \\ \hline 
\end{tabular}
\\[12pt]
\begin{tabular}{c|cccc}\hline
$S(231)$ & $3124$ & $4312$ & $4132$ & $1324$  \\ \hline
Dyck path  
&  
\tikzpic{-0.5}{[scale=0.3, show background rectangle,background rectangle/.style={fill=none},inner frame sep=1mm]
\draw(0,0)--(3,3)--(5,1)--(6,2)--(8,0);
\draw[dashed](1,1)--(2,0)--(4,2)(2,2)--(4,0)--(5,1)--(6,0)--(7,1);
}
&\tikzpic{-0.5}{[scale=0.3, show background rectangle,background rectangle/.style={fill=none},inner frame sep=1mm]
\draw(0,0)--(2,2)--(3,1)--(4,2)--(5,1)--(6,2)--(8,0);
\draw[dashed](1,1)--(2,0)--(3,1)--(4,0)--(5,1)--(6,0)--(7,1);
}
& \tikzpic{-0.5}{[scale=0.3, show background rectangle,background rectangle/.style={fill=none},inner frame sep=1mm]
\draw(0,0)--(1,1)--(2,0)--(4,2)--(5,1)--(6,2)--(8,0);
\draw[dashed](3,1)--(4,0)--(5,1)--(6,0)--(7,1);
}
& 
\tikzpic{-0.5}{[scale=0.3, show background rectangle,background rectangle/.style={fill=none},inner frame sep=1mm]
\draw(0,0)--(2,2)--(4,0)--(6,2)--(8,0);
\draw[dashed](1,1)--(2,0)--(3,1)(5,1)--(6,0)--(7,1);
}
 \\ \hline
Dyck tiling 
&\tikzpic{-0.5}{[scale=0.3, show background rectangle,background rectangle/.style={fill=none},inner frame sep=1mm]
\draw(0,0)--(1,1)--(2,0)--(3,1)--(4,0)--(5,1)--(6,0)--(7,1)--(8,0);
\draw(0,0)--(-1,1)--(0,2)--(-1,3)--(0,4)--(-1,5)--(0,6)--(-1,7)--(0,8)--(7,1);
\draw(7,0)node{$3$}(5,0)node{$1$}(3,0)node{$2$}(1,0)node{$4$};
\draw(1,1)--(0,2)(5,3)--(4,2)--(3,3)--(2,2)--(1,1);
\draw(0,6)--(3,3)--(4,4)(0,4)--(2,2)(0,6)--(1,7)(0,4)--(2,6)(0,2)--(3,5);
\draw(4,1.3)node{$5$};
} 
&
\tikzpic{-0.5}{[scale=0.3, show background rectangle,background rectangle/.style={fill=none},inner frame sep=1mm]
\draw(0,0)--(1,1)--(2,0)--(3,1)--(4,0)--(5,1)--(6,0)--(7,1)--(8,0);
\draw(0,0)--(-1,1)--(0,2)--(-1,3)--(0,4)--(-1,5)--(0,6)--(-1,7)--(0,8)--(7,1);
\draw(7,0)node{$4$}(5,0)node{$3$}(3,0)node{$1$}(1,0)node{$2$};
\draw(5,3)--(4,2)--(3,3)--(2,2)--(1,3)--(0,2)(3,5)--(2,4)--(1,5)--(0,4);
\draw(0,6)--(1,7)(0,6)--(1,5)--(2,6);
\draw(3,1.8)node{$4$}(2,3.3)node{$4$};
} 
&
\tikzpic{-0.5}{[scale=0.3, show background rectangle,background rectangle/.style={fill=none},inner frame sep=1mm]
\draw(0,0)--(1,1)--(2,0)--(3,1)--(4,0)--(5,1)--(6,0)--(7,1)--(8,0);
\draw(0,0)--(-1,1)--(0,2)--(-1,3)--(0,4)--(-1,5)--(0,6)--(-1,7)--(0,8)--(7,1);
\draw(7,0)node{$4$}(5,0)node{$1$}(3,0)node{$3$}(1,0)node{$2$};
\draw(5,3)--(4,2)--(3,3)--(2,2)--(1,3)--(0,2)(3,3)--(1,5)--(0,4);
\draw(0,6)--(1,7)(0,6)--(1,5)--(2,6)(2,4)--(3,5)(3,3)--(4,4);
\draw(3,1.8)node{$4$}(1,3.5)node{$3$};	
} 
&
\tikzpic{-0.5}{[scale=0.3, show background rectangle,background rectangle/.style={fill=none},inner frame sep=1mm]
\draw(0,0)--(1,1)--(2,0)--(3,1)--(4,0)--(5,1)--(6,0)--(7,1)--(8,0);
\draw(0,0)--(-1,1)--(0,2)--(-1,3)--(0,4)--(-1,5)--(0,6)--(-1,7)--(0,8)--(7,1);
\draw(7,0)node{$1$}(5,0)node{$3$}(3,0)node{$2$}(1,0)node{$4$};
\draw(1,1)--(0,2)(5,1)--(3,3)--(1,1)(2,2)--(0,4)(0,6)--(1,7)(0,6)--(3,3)--(4,4);
\draw(0,4)--(2,6)(0,2)--(3,5)(5,1)--(6,2)(4,2)--(5,3);
\draw(3,1.8)node{$4$};
} 
 \\ \hline
\end{tabular}
\\[12pt]
\begin{tabular}{c|cccc}\hline
$S(231)$ & $1432$ & $2134$ & $4213$ & $2143$ \\ \hline
Dyck path 
&  
\tikzpic{-0.5}{[scale=0.3, show background rectangle,background rectangle/.style={fill=none},inner frame sep=1mm]
\draw(0,0)--(1,1)--(2,0)--(3,1)--(4,0)--(6,2)--(8,0);
\draw[dashed](5,1)--(6,0)--(7,1);
}
&
\tikzpic{-0.5}{[scale=0.3, show background rectangle,background rectangle/.style={fill=none},inner frame sep=1mm]
\draw(0,0)--(3,3)--(6,0)--(7,1)--(8,0);
\draw[dashed](1,1)--(2,0)--(4,2)(2,2)--(4,0)--(5,1);
}
&
\tikzpic{-0.5}{[scale=0.3, show background rectangle,background rectangle/.style={fill=none},inner frame sep=1mm]
\draw(0,0)--(2,2)--(3,1)--(4,2)--(6,0)--(7,1)--(8,0);
\draw[dashed](1,1)--(2,0)--(3,1)--(4,0)--(5,1);
}
&
\tikzpic{-0.5}{[scale=0.3, show background rectangle,background rectangle/.style={fill=none},inner frame sep=1mm]
\draw(0,0)--(1,1)--(2,0)--(4,2)--(6,0)--(7,1)--(8,0);
\draw[dashed](3,1)--(4,0)--(5,1);
} \\ \hline
Dyck tilig 
&\tikzpic{-0.5}{[scale=0.3, show background rectangle,background rectangle/.style={fill=none},inner frame sep=1mm]
\draw(0,0)--(1,1)--(2,0)--(3,1)--(4,0)--(5,1)--(6,0)--(7,1)--(8,0);
\draw(0,0)--(-1,1)--(0,2)--(-1,3)--(0,4)--(-1,5)--(0,6)--(-1,7)--(0,8)--(7,1);
\draw(7,0)node{$1$}(5,0)node{$4$}(3,0)node{$3$}(1,0)node{$2$};
\draw(5,1)--(3,3)--(2,2)--(1,3)--(0,2)(3,3)--(1,5)--(0,4);
\draw(0,6)--(1,7)(0,6)--(1,5)--(2,6)(2,4)--(3,5)(3,3)--(4,4)(4,2)--(5,3)(5,1)--(6,2);
\draw(2,1.3)node{$3$}(1,3.8)node{$3$};
} 
&
\tikzpic{-0.5}{[scale=0.3, show background rectangle,background rectangle/.style={fill=none},inner frame sep=1mm]
\draw(0,0)--(1,1)--(2,0)--(3,1)--(4,0)--(5,1)--(6,0)--(7,1)--(8,0);
\draw(0,0)--(-1,1)--(0,2)--(-1,3)--(0,4)--(-1,5)--(0,6)--(-1,7)--(0,8)--(7,1);
\draw(7,0)node{$2$}(5,0)node{$1$}(3,0)node{$3$}(1,0)node{$4$};
\draw(1,1)--(0,2)(3,1)--(0,4)(5,3)--(3,1)(0,6)--(1,7)(0,6)--(4,2)(0,4)--(2,6)(0,2)--(3,5)(1,1)--(4,4);
\draw(5,1.8)node{$6$};
} 
&
\tikzpic{-0.5}{[scale=0.3, show background rectangle,background rectangle/.style={fill=none},inner frame sep=1mm]
\draw(0,0)--(1,1)--(2,0)--(3,1)--(4,0)--(5,1)--(6,0)--(7,1)--(8,0);
\draw(0,0)--(-1,1)--(0,2)--(-1,3)--(0,4)--(-1,5)--(0,6)--(-1,7)--(0,8)--(7,1);
\draw(7,0)node{$4$}(5,0)node{$2$}(3,0)node{$1$}(1,0)node{$3$};
\draw(5,3)--(4,2)--(3,3)--(2,2)--(1,3)--(0,2)(1,3)--(0,4)(3,5)--(1,3)(0,6)--(1,7)(0,6)--(2,4)(0,4)--(2,6);
\draw(3,1.8)node{$4$}(3,3.8)node{$5$};
} 
&
\tikzpic{-0.5}{[scale=0.3, show background rectangle,background rectangle/.style={fill=none},inner frame sep=1mm]
\draw(0,0)--(1,1)--(2,0)--(3,1)--(4,0)--(5,1)--(6,0)--(7,1)--(8,0);
\draw(0,0)--(-1,1)--(0,2)--(-1,3)--(0,4)--(-1,5)--(0,6)--(-1,7)--(0,8)--(7,1);
\draw(7,0)node{$2$}(5,0)node{$1$}(3,0)node{$4$}(1,0)node{$3$};
\draw(3,1)--(1,3)--(0,2)(1,3)--(0,4)(5,3)--(3,1)(0,6)--(1,7)(0,6)--(4,2)(0,4)--(2,6)(1,3)--(3,5)(2,2)--(4,4);
\draw(5,1.8)node{$6$}(1,1.8)node{$2$};
} 
 \\ \hline
\end{tabular}
\\[12pt]
\begin{tabular}{c|cc}\hline
$S(231)$ & $3214$ & $4321$  \\ \hline
Dyck path 
&
\tikzpic{-0.5}{[scale=0.3, show background rectangle,background rectangle/.style={fill=none},inner frame sep=1mm]
\draw(0,0)--(2,2)--(4,0)--(5,1)--(6,0)--(7,1)--(8,0);
\draw[dashed](1,1)--(2,0)--(3,1);
} 
&
\tikzpic{-0.5}{[scale=0.3, show background rectangle,background rectangle/.style={fill=none},inner frame sep=1mm]
\draw(0,0)--(1,1)--(2,0)--(3,1)--(4,0)--(5,1)--(6,0)--(7,1)--(8,0);
} \\ \hline
Dyck tiling 
&\tikzpic{-0.5}{[scale=0.3, show background rectangle,background rectangle/.style={fill=none},inner frame sep=1mm]
\draw(0,0)--(1,1)--(2,0)--(3,1)--(4,0)--(5,1)--(6,0)--(7,1)--(8,0);
\draw(0,0)--(-1,1)--(0,2)--(-1,3)--(0,4)--(-1,5)--(0,6)--(-1,7)--(0,8)--(7,1);
\draw(7,0)node{$3$}(5,0)node{$2$}(3,0)node{$1$}(1,0)node{$4$};
\draw(1,1)--(0,2)(5,3)--(4,2)--(3,3)--(1,1)(2,2)--(0,4)(3,5)--(0,2)(0,6)--(1,7)(0,6)--(2,4)(0,4)--(2,6);
\draw(4,1.3)node{$5$}(3,3.8)node{$5$};
} 
&
\tikzpic{-0.5}{[scale=0.3, show background rectangle,background rectangle/.style={fill=none},inner frame sep=1mm]
\draw(0,0)--(1,1)--(2,0)--(3,1)--(4,0)--(5,1)--(6,0)--(7,1)--(8,0);
\draw(0,0)--(-1,1)--(0,2)--(-1,3)--(0,4)--(-1,5)--(0,6)--(-1,7)--(0,8)--(7,1);
\draw(7,0)node{$4$}(5,0)node{$3$}(3,0)node{$2$}(1,0)node{$1$};
\draw(5,3)--(4,2)--(3,3)--(2,2)--(1,3)--(0,2)(3,5)--(2,4)--(1,5)--(0,4)(1,7)--(0,6);
\draw(3,1.8)node{$4$}(2,3.3)node{$4$}(1,5.8)node{$4$};
} 
\\ \hline
\end{tabular}
\caption{The bijection among $231$-avoiding permutations, Dyck paths and Dyck tilings}
\label{fig:bij}
\end{figure}

\begin{example}
In Figure \ref{fig:bij}, we summarize the correspondence among $231$-avoiding permutations, Dyck paths and 
corresponding Dyck tilings for $n=4$.
The integer in a non-trivial Dyck tile is the major index of this tile.
\end{example}

\subsection{Dyck tilings and \texorpdfstring{$q$}{q}-Catalan number \texorpdfstring{$C_{n}(q)$}{Cn(q)}}
Let $w\in S(231)$ and $D(w)$ be the corresponding Dyck tiling.
We define the weight $\mathrm{wt}'(d)$ of a Dyck tile $d$ of size $k$, $0\le k$, as 
\begin{align*}
\mathrm{wt}'(d):=k.
\end{align*}
Similarly, the weight $\mathrm{wt}'(D)$ of a Dyck tiling $D$ is given by 
\begin{align*}
\mathrm{wt}'(D):=\sum_{d\in D}\mathrm{wt}'(d).
\end{align*}

\begin{prop}
\label{prop:DTqCat}
We have 
\begin{align*}
\sum_{w\in S(231)}q^{\mathrm{wt}'(D(w))}=q^{\genfrac{(}{)}{0pt}{}{n}{2}}C_{n}(q^{-1}),
\end{align*}
where $C_{n}(q)$ is defined in Eq. (\ref{eq:defCn}).
\end{prop}
\begin{proof}
We consider Dyck tilings in the region $R$. We have a non-trivial Dyck tile $d$ if 
there is an inversion and the size of $d$ is equal to the inversion number.
From these, we have 
\begin{align*}
\sum_{w\in S(231)}q^{\mathrm{wt}'(D(w))}=\sum_{w\in S(231)}q^{\mathrm{Inv}(w)}.
\end{align*}
The right hand side of the above equation gives $q^{\genfrac{(}{)}{0pt}{}{n}{2}}C_{n}(q^{-1})$
by Proposition \ref{prop:231Crn} with $r=1$.
This completes the proof.
\end{proof}

\section{Symmetric Dyck paths and \texorpdfstring{$q$}{q}-Catalan number}
\label{sec:SymDP}
\subsection{Symmetric Dyck paths}
In this section, we interpret the $q$-Catalan numbers in terms of symmetric Dyck paths.
For this, we introduce the bijection $\Psi$ between a Dyck path of size $n$ and a symmetric 
Dyck path of size $n-1$ following \cite[Section 7.3]{S23}.

Recall that a Dyck path $\mu:=\mu_1\mu_2\ldots\mu_{n}$ of size $n$ is expressed as a word of 
two alphabets $\{0,1\}$.
We define the operation $\rho:\mathcal{P}^{1}_{r}\rightarrow\mathcal{P}^{1}_{r}$, 
$\mu\mapsto\overline{\mu_{n}}\overline{\mu_{n-1}}\ldots\overline{\mu_1}$ where 
$\overline{0}=1$ and $\overline{1}=0$.

\begin{defn}
A Dyck path of length $n$ is said to be symmetric if it satisfies 
$\mu=\rho(\mu)$.
\end{defn}

It is obvious that the length of a symmetric Dyck path is always even.
Since a symmetric Dyck path $\mu:=\mu_{1}\ldots \mu_{2n}$ is invariant 
under the action of $\rho$, we identify $\mu$ with a word $\nu:=\mu_1\ldots\mu_{n}$ of length $n$.
Note that we can recover $\mu$ from $\nu$ uniquely.

We introduce a notion of arches and half-arches for a symmetric Dyck path $\mu=\mu_1\ldots\mu_{2n}$.
Suppose that $\mu_i=0$ for $1\le i\le n$. Then, let $j>i$ be the minimal integer such that 
$\mu_{j}=1$ and the sub-word $\mu_i\mu_{i+1}\ldots\mu_{j}$ is a prime Dyck path.
The pair $(i,j)$ is called an arch if $j\le n$.
Similarly, if the pair $(i,j)$ or equivalently the integer $i$, is called a half-arch 
if $j\ge n+1$. 

\begin{defn}
The size of a symmetric Dyck path is the sum of the numbers of arches and half-arches.
We denote by $\mathcal{Q}_{n}$ the set of symmetric Dyck paths of size $n$. 
\end{defn}

For example, we have five symmetric Dyck paths of size two:
\begin{align*}
00 \quad 001 \quad 010  \quad 0011 \quad 0101 
\end{align*}
Note that the size of a symmetric Dyck path does not coincide with the length of the word in $\{0,1\}^{\ast}$.

In what follows, we introduce the bijection $\Psi$ between a Dyck path in $\mathcal{P}^{1}_{n}$ and 
a symmetric Dyck path in $\mathcal{Q}_{n-1}$ following \cite{S23}.
We first introduce a map $\Psi:\mathcal{P}^{1}_{n}\rightarrow\mathcal{Q}_{n-1}$, 
$\mu\mapsto\nu$ by the following algorithm:
\begin{enumerate}
\item Delete the first $0$ and the last $1$ in $\mu$.
\item If the first alphabet is $1$ and it is followed by $0$, we replace this pair of $1$ and $0$ 
with a single alphabet $0$.
\item If the first alphabet is $0$, take a maximal prime Dyck path starting from this $0$.
Then, we leave the alphabets in this maximal prime Dyck path as they are.
\item We apply (2) and (3) to the remaining word until we replace all the alphabets with 
new alphabets.
\end{enumerate}

We define the map $\Psi^{-1}:\mathcal{Q}_{n}\rightarrow\mathcal{P}^1_{n+1}$, $\nu\mapsto\mu$ by 
the following algorithm \cite[Section 7.3]{S23}:
\begin{enumerate}
\item Replace a half-arch in $\nu$ which consists of single alphabet $0$ by $10$.
\item Keep an arch which consists of a pair of $0$ and $1$ as it is.
\item After replacing all arches with new alphabets, add $0$ from left and $1$ from right.
\end{enumerate}

The maps $\Psi$ and $\Psi^{-1}$ are obviously inverse to each other, therefore the map 
$\Psi$ is a bijection $\mathcal{P}^{1}_{n+1}\xrightarrow{\sim}\mathcal{Q}_{n}$.

\subsection{Major index for symmetric Dyck paths}
The descent set $\mathrm{Des}(\nu)$ of $\nu$ is the set of integers $1\le i\le n-1$ such that $(\nu_{i},\nu_{i+1})=(1,0)$ and 
the integer $n$ if $\nu_{n}=1$.
Note that the integer $n$ is not included in the descent set if $\nu_{n}=0$. 
We define the major index on a symmetric Dyck path $\nu\in\mathcal{Q}_{n}$ as follows:
\begin{align*}
\mathrm{maj}(\nu):=\sum_{i\in\mathrm{Des}(\nu)}i.
\end{align*}

\begin{theorem}
\label{thrm:majsymD}
The major index on $\mathcal{Q}_{n}$ satisfies 
\begin{align*}
\sum_{\nu\in\mathcal{Q}_{n}}q^{\mathrm{maj}(\nu)}
=\genfrac{}{}{}{}{1}{[n+2]}\genfrac{[}{]}{0pt}{}{2(n+1)}{n+1}.
\end{align*}
\end{theorem}

\begin{example}
We have fourteen Dyck paths in $\mathcal{P}^{1}_{4}$, and fourteen 
symmetric Dyck paths in $\mathcal{Q}_{3}$.
\begin{table}[ht]
\begin{tabular}{c|ccccc}\hline 
Dyck path $\sigma$ & $00001111$ & $0001.0111$ & $00011.011$ & $000111.01$ & $001.00111$  \\ \hline 
$\mathrm{maj}(\sigma)$ & $0$ & $4$ & $5$ & $6$ & $3$  \\ 
$\Psi(\sigma)$ & $000111$ & $001.011$ & $0011.01$ & $0011.0$ & $01.0011$ \\ 
$\mathrm{maj}(\Psi(\sigma))$ & $6$ & $9$ & $10$ & $4$ & $8$ \\ \hline 
\end{tabular}
\\[12pt]
\begin{tabular}{c|ccccc}\hline 
$\sigma$ & $001.01.011$ & $001.011.01$ & $0011.0011$ & $0011.01.01$ & $01.000111$   \\ \hline 
$\mathrm{maj}(\sigma)$  & $8$ & $9$ & $4$  & $10$ & $2$ \\ 
$\Psi(\sigma)$& $01.01.01$ & $01.01.0$ & $01.001$  & $01.00$ & $00011$ \\ 
$\mathrm{maj}(\Psi(\sigma))$ & $12$ & $6$ & $7$ & $2$ & $5$ \\ \hline 
\end{tabular}
\\[12pt]
\begin{tabular}{c|cccc}\hline 
$\sigma$ &  $01.001.011$ & $01.0011.01$  & $01.01.0011$ & $01.01.01.01$\\ \hline 
$\mathrm{maj}(\sigma)$  & $7$ & $8$ & $6$ & $12$ \\ 
$\Psi(\sigma)$  & $001.01$ & $001.0$ & $0001$ & $000$  \\ 	
$\mathrm{maj}(\Psi(\sigma))$  & $8$ & $3$ & $4$ & $0$ \\ \hline 
\end{tabular}
\\[12pt]
\caption{The bijection between Dyck paths and symmetric Dyck paths}
\label{table:bijDsymD}
\end{table}
In Table \ref{table:bijDsymD}, we summarize the Dyck paths, their 
major indices, corresponding symmetric Dyck paths, and their major 
indices.
\end{example}

Before proceeding to the proof of Theorem \ref{thrm:majsymD}, we introduce 
two lemmas.

Let $\mathcal{Q}_{n}(r)$ be the set of symmetric Dyck paths of size $n$ which 
have $r$ descents.
We define 
\begin{align*}
N_{3}(n,r):=\sum_{\nu\in\mathcal{Q}_{n}(r)}q^{\mathrm{maj}(\nu)}.
\end{align*}
The generating function $N_{3}(n,r)$ satisfies the following recurrence relation.
\begin{lemma}
We have 
\begin{align}
\label{eq:recN3}
N_{3}(n,r)=q^{r}N_{3}(n-1,r)+\sum_{j=1}^{n}\sum_{s=0}^{j-1}N(j-1,s)N_{3}(n-j,r-s-1)q^{2j(r-s)+s},
\end{align}
where $N(n,r)$ is the $q$-Narayana number defined in Eq. (\ref{def:qNara}).
\end{lemma}
\begin{proof}
Since $\nu:=\nu_1\ldots\nu_{m}\in\mathcal{Q}_{n}(r)$, the first letter $\nu_1$ is $0$.
We consider two cases: 1) $\nu_1$ forms a half-arch, and 2) $\nu_{1}$ forms
an arch with some $2j$ such that $\nu_{2j}=1$.

In the first case, the word $\nu':=\nu_2\ldots\nu_m$ is in $\mathcal{Q}_{n-1}(r)$. Further,
since each descent in $\nu'$ is shifted to right by one unit in $\nu$, we have a factor 
$q^{r}$. This implies that the words $\nu'$ contribute to $N_3(n,r)$ as $q^{r}N_{3}(n-1,r)$.

In the second case, the sub-path $\nu_1\ldots\nu_{2j}$ is a prime Dyck path of size $j$.
Since to be prime means that $\nu_{2j}=1$ and $\nu_2=0$, the sub-path $\nu_2\ldots\nu_{2j-1}$ is 
a Dyck path of size $j-1$. This gives the factor $N(j-1,s)$ and $q^{s}$ if $\nu_2\ldots\nu_{2j-1}$
has $s$ descents.
The sub-path $\nu_{2j+1}\ldots\nu_{m}$ is in $\mathcal{Q}_{n-j}(r-s-1)$, and shifted rightward
by $2j$ unit. This gives a generating function $N_3(n-j,r-s-1)$ and a factor $q^{2j(r-s-1)}$. 
Finally, $\nu_{2j}$ is a descent in $\nu$ since $\nu_{2j}=1$ and $\nu_{2j+1}=0$. This gives a factor $q^{2j}$.
The contribution of $\nu$ to $N_{3}(n,r)$ is given by the product of all contributions, which 
is $N(j-1,s)N_3(n-j,r-s-1)q^{2j(r-s)+s}$.
By taking the sum over $j$ and $s$, we obtain the expression (\ref{eq:recN3}), which completes 
the proof.
\end{proof}

\begin{lemma}
\label{lemma:N3}
We have $N_3(n,r)=N(n+1,r)$.
\end{lemma}
\begin{proof}
We compare the expression (\ref{eq:recN}) with Eq. (\ref{eq:recN3}).
Then, it is obvious that we have $N_{3}(n,r)=N(n+1,r)$.
\end{proof}

\begin{proof}[Proof of Theorem \ref{thrm:majsymD}]
We have 
\begin{align*}
\sum_{\nu\in\mathcal{Q}_{n}}q^{\mathrm{maj}(\nu)}
&=\sum_{r=0}^{n}N_{3}(n,r)=\sum_{r=0}^{n}N(n+1,r), \\
&=\genfrac{}{}{}{}{1}{[n+2]}\genfrac{[}{]}{0pt}{}{2(n+1)}{n+1},
\end{align*}
where we have used Lemma \ref{lemma:N3} and the property of the $q$-Narayana numbers.
\end{proof}

We refine the $q$-Narayana numbers for symmetric Dyck paths as follows.
Let $\mathcal{Q}(n,r,k)$ be the set of symmetric Dyck paths which are 
size $n$, have $r$ descents and $k$ half-arches.
Then, we define 
\begin{align*}
N_{3}(n,r,k):=\sum_{\nu\in\mathcal{Q}(n,r,k)}q^{\mathrm{maj}(\nu)}.
\end{align*}

\begin{theorem}
\label{thrm:N3nrk}
We have 
\begin{align}
\label{eq:exN3}
N_{3}(n,r,k)=\genfrac{[}{]}{0pt}{}{n}{k+1}^{-1}\genfrac{[}{]}{0pt}{}{n}{r}
\genfrac{[}{]}{0pt}{}{n}{r-1}\genfrac{[}{]}{0pt}{}{n-r}{k}q^{r(r+1)+2(n-r-k)}.
\end{align}
\end{theorem}

We prove Theorem \ref{thrm:N3nrk} after introducing two lemmas which give
recurrence relations for the generating function $N_{3}(n,r,k)$.

\begin{lemma}
The generating function $N_{3}(n,r,k)$ has the following recurrence relation
\begin{align}
\label{eq:recN3inN}
N_{3}(n,r,k)=q^{r}N_{3}(n-1,r,k-1)
+\sum_{j=1}^{n}\sum_{s=0}^{j-1}N(j-1,s)N_{3}(n-j,r-s-1,k)q^{2j(r-s)+s}.
\end{align}
\end{lemma}
\begin{proof}
Since $\nu:=\nu_1\ldots\nu_{m}\in\mathcal{Q}(n,r,k)$, the first letter in $\nu$ is $0$.
We have two cases for the first $0$ in $\nu$: 1) this letter forms 
a half-arch, and 2) this letter forms an arch with some $1$ in $\nu$.

We consider the first case. The word $\nu'=\nu_2\ldots \nu_{m}$ is in $\mathcal{Q}(n-1,r,k-1)$ 
since the first letter $\nu_2$ of $\nu'$ is $0$. The existence of the first $0$ in $\nu$ 
gives the factor $q^{r}$ and such $\nu'$ contributes to $N_3(n,r,k)$ as $q^{r}N_{3}(n-1,r,k-1)$.

For the second case, let $2j\in[1,m]$ be the integer such that $\nu_{2j}=1$ and this $1$ forms 
an arch with $\nu_{1}=0$. Then, it is obvious that a sub-word $\nu_1\ldots \nu_{2j}$ is 
a prime Dyck path of size $j$.
If we divide the word $\nu$ into two sub-words $\nu'=\nu_1\ldots\nu_{2j}$ and $\nu''=\nu_{2j+1}\ldots \nu_{m}$,
$\nu$ is expressed as a concatenation of $\nu'$ and $\nu''$.
The property that $\nu'$ is prime insures that this expression is unique.
Since $\nu'$ is prime, $\nu'$ contribute to $N_{3}(n,r,k)$ as $q^{s}N(j-1,s)$ where $s$ is the number of valleys
in $\nu'$. The word $\nu''$ is in $\mathcal{Q}(n-j,r-s-1,k)$. Since we have $2j$ letters left to $\nu''$ in $\nu$,
we have a factor $q^{2j(r-s-1)}$. Further, since the last letter of $\nu'$ is $1$ and the first letter of $\nu''$=0,
the integer $2j$ is a descent in $\nu$. This gives the factor $q^{2j}$.
The concatenation of two words implies that the contribution of $\nu$ to $N_3(n,r,k)$ is the product of 
all these contributions. This gives the term $N(j-1,s)N_{3}(n-j,r-s-1,k)q^{2j(r-s)+s}$.
If we sum over all possible $j$ and $s$, we obtain the expression (\ref{eq:recN3inN}), which 
completes the proof.
\end{proof}

The last alphabet of a symmetric Dyck path $\nu=\nu_1\ldots\nu_{n}$ in 
$\mathcal{Q}(n,r,k)$ is either $0$ or $1$.
We denote by $\mathcal{Q}_{+}(n,r,k)$ (resp. $\mathcal{Q}_{-}(n,r,k)$)
the set of symmetric Dyck path of size $n$ with the last alphabet $0$ (resp. $1$).
Then, it is obvious that we have 
$\mathcal{Q}(n,r,k)=\mathcal{Q}_{+}(n,r,k)\cup\mathcal{Q}_{-}(n,r,k)$. 
We further refine the $N_{3}(n,r,k)$ as follows.
We define 
\begin{align*}
N_{3,\epsilon}(n,r,k):=\sum_{\nu\in\mathcal{Q}_{\epsilon}(n,r,k)}q^{\mathrm{maj}(\nu)},
\end{align*}
where $\epsilon=\pm$.

\begin{lemma}
\label{lemma:recNpm}
The generating functions $N_{3,\epsilon}(n,r,k)$, $\epsilon=\pm$, satisfy the 
following recurrence relation:
\begin{align}
\label{eq:N3sum}
N_{3}(n,r,k)&=N_{3,+}(n,r,k)+N_{3,-}(n,r,k), \\
\label{eq:N3p}
N_{3,+}(n,r,k)&=N_{3}(n-1,r,k-1), \\
\label{eq:N3m}
N_{3,-}(n,r,k)&=q^{2n-k}\cdot N_{3}(n-1,r-1,k)+q\cdot  N_{3,-}(n,r,k+1).
\end{align}
\end{lemma}
\begin{proof}
The definition of $Q_{\epsilon}(n,r,k)$ implies Eq. (\ref{eq:N3sum}).
If $\nu=\nu_1\ldots\nu_{m}\in\mathcal{Q}_{+}(n,r,k)$, the last alphabet of $\nu$
is $0$. This means that this alphabet is a half-arch. The word $\nu'=\nu_1\ldots\nu_{m-1}$
belongs to $\mathcal{Q}(n-1,r,k-1)$, which implies Eq. (\ref{eq:N3p}).

When $\nu\in\mathcal{Q}_{-}(n,r,k)$, we have two cases: $\nu_{m-1}=0$ and $\nu_{m-1}=1$.
In the former case, we have an arch which consists of $\nu_{m-1}=0$ and $\nu_{m}=1$.
Then, the word $\nu=1\ldots\nu_{m-2}$ belongs to $\mathcal{Q}(n-1,r-1,k)$.
Since $\nu_{m}=1$ implies that $m$ is in the descent set. This $\nu_{m}=1$ contributes 
to the generating function by $q^{2n-k}$.
Similarly, when $\nu_{m-1}=1$, the word $\nu''=\nu_1\ldots\nu_{m-1}$ belong to 
$\mathcal{Q}_{-}(n,r,k+1)$.  Note that the integer $m-1$ is a descent for $\nu''$, and 
$m$ is a descent for $\nu$ but $m-1$ is not a descent for $\nu$.
This gives the factor $q=q^{m-(m-1)}$. As a consequence, we have Eq. (\ref{eq:N3m}).
\end{proof}

\begin{proof}[Proof of Theorem \ref{thrm:N3nrk}]
From Lemma \ref{lemma:recNpm}, the generating function $N_{3}(n,r,k)$ satisfies
\begin{align*}
\begin{aligned}
&N_{3}(n,r,k)-q\cdot N_{3}(n,r,k+1)-N_{3}(n-1,r,k-1) \\
&\qquad+q\cdot N_{3}(n-1,r,k)-q^{2n-k}\cdot N_{3}(n-1,r-1,k)=0.
\end{aligned}
\end{align*}
It is clear that the expression (\ref{eq:exN3}) satisfies the 
recurrence relation.
\end{proof}

\begin{cor}
We have 
\begin{align}
\label{eq:N3sum2}
\sum_{k=0}^{n}N_{3}(n,r,k)=N(n,r).
\end{align}
\end{cor}
\begin{proof}
Let $N'(n,r)$ be the left hand side of Eq. (\ref{eq:N3sum2}).
If we take the sum of Eq. (\ref{eq:recN3inN}) with respect to $k$,
we have 
\begin{align*}
N'(n,r)=q^{r}N'(n-1,r)+\sum_{j=1}^{n}\sum_{s=0}^{j-1}N(j-1,s)N'(n-j,r-s-1)q^{2j(r-s)+s}.
\end{align*}
By comparing the above equation with Eq. (\ref{eq:recN}), it is obvious that 
$N'(n,r)=N(n,r)$ which is the $q$-Narayana number.
\end{proof}

\subsection{Symmetric Dyck paths and \texorpdfstring{$q$}{q}-Catalan number \texorpdfstring{$C_{n}(q)$}{Cn(q)}}
Let $\nu=\nu_{1}\ldots\nu_{m}$ be a symmetric Dyck path in $\mathcal{Q}_{n}$.
Recall that $\nu$ consists of $n-k$ arches $a_{i}$, $1\le i\le n-k$, and $k$ 
half-arches $b_{j}$, $1\le j\le k$, with $0\le k\le n$. 
Then, we define $\mathrm{Inv}_{1}(\nu)$ by
\begin{align*}
\mathrm{Inv}_{1}(\nu):=k+\sum_{j=1}^{k}\#\{a_{i} : \text{$a_{i}$ is right to $b_{j}$ in $\nu$}\}
+\sum_{j=1}^{k}\#\{b_{i}: \text{$b_i$ is right to $b_{j}$ in $\nu$}\}.
\end{align*}
The inversion number $\mathrm{Inv}_2(\nu)$ is defined to be 
\begin{align*}
\mathrm{Inv}_2(\nu):=\#\{(i,j): (\nu_{i},\nu_{j})=(1,0) \text{ with } 1\le i<j\le m \}.
\end{align*}
We define the inversion number $\mathrm{Inv}(\nu)$ of $\nu$ by
\begin{align*}
\mathrm{Inv}(\nu):=\mathrm{Inv}_1(\nu)+\mathrm{Inv}_{2}(\nu).
\end{align*}

\begin{prop}
\label{prop:symDPqCat}
We have
\begin{align}
\sum_{\nu\in\mathcal{Q}_{n}}q^{\mathrm{Inv}(\nu)}=C_{n+1}(q).
\end{align}
\end{prop}
\begin{proof}
Recall that $\Psi: \mathcal{P}^1_{n}\rightarrow\mathcal{Q}_{n-1}$ is a bijection.
To prove the statement, it is enough to show that $\mathrm{Inv}(\nu)=\mathrm{Inv}(\Psi^{-1}(\nu))$
for $\nu\in\mathcal{Q}_{n}$ and $\Psi^{-1}(\nu)\in\mathcal{P}^1_{n+1}$ since we have
\begin{align*}
\sum_{\mu\in\mathcal{P}^1_{n+1}}q^{\mathrm{Inv}(\mu)}=C_{n+1}(q),
\end{align*}
by Eq. (\ref{eq:AreaCn}).
The bijection $\Psi^{-1}$ maps a half-arch $0$ to a sub-sequence $10$. We consider the inversion number
relative to this sub-sequence $10$. The inversion number $\mathrm{Inv}_{1}(\nu)$ counts the inversion number
of $1$ in $10$, and $\mathrm{Inv}_{2}(\nu)$ counts the inversion number of $0$ in $10$.
These imply that $\mathrm{Inv}(\nu)=\mathrm{Inv}(\Psi^{-1}(\nu))$.
This completes the proof.
\end{proof}

\bibliographystyle{amsplainhyper} 
\bibliography{biblio}

\end{document}